\documentclass{imatma}
\jno{xxx000}
\received{}
\revised{}
\usepackage{graphicx,amssymb,amsmath}
\usepackage{amsthm}

\usepackage{bm}
\usepackage{authblk}

\usepackage{natbib} 
\usepackage{xpatch}

\newcommand{\myAd}{\mathcal{A}\!\!d}

\usepackage{amsfonts}
\newcommand{\upL}{\mathsf{L}}
\newcommand{\upd}{\mathrm{d}}
\newcommand{\ds}{\displaystyle}
\newcommand{\bphi}{\bm{\phi}}
\newcommand{\bkappa}{\bm{\kappa}}
\newcommand{\bom}{\bm{\sigma}}
\newcommand{\euu}{\mathrm{E}_u}
\newcommand{\eux}{\mathrm{E}_x}
\newcommand{\eubu}{\mathrm{E}_\mbu}
\newcommand{\eual}{\mathrm{E}_{\displaystyle{u^\alpha}}}
\newcommand{\euka}{\mathrm{E}_\kappa}

\newcommand{\eumu}{\mathrm{E}_\mu}
\newcommand{\eunu}{\mathrm{E}_\nu}
\newcommand{\eubka}{\mathrm{E}_{\bkappa}}
\newcommand{\euet}{\mathrm{E}_\eta}
\newcommand{\ode}{O$\Delta$E }

\newcommand{\odes}{O$\Delta$Es }
\newcommand{\odese}{O$\Delta$Es}
\newcommand{\es}{\mathrm{S}}
\newcommand{\iup}{\iota}
\newcommand{\id}{\mathrm{id}}
\newcommand{\mba}{\mathbf{a}}
\newcommand{\mbu}{\mathbf{u}}
\newcommand{\mbv}{\mathbf{v}}
\newcommand{\mbw}{\mathbf{w}}

\newcommand{\man}{\mathcal{M}}
\newcommand{\hth}{h_{\theta}}
\newcommand{\co}{\cos\hth}
\newcommand{\si}{\sin\hth}
\newcommand{\eul}{\mathrm{E}_{\ell}}
\newcommand{\eut}{\mathrm{E}_{\hth}}

\begin{document}
\title{Moving Frames and Noether's  Finite Difference Conservation Laws I.}
\shorttitle{Noether's Finite Difference Conservation Laws I}
\author{  %
 {\sc 
 E. L. Mansfield\thanks{Corresponding author. Email: e.l.mansfield@kent.ac.uk},  
A. Rojo-Echebur\'{u}a\thanks{Email: arer2@kent.ac.uk} and 
P. E. Hydon}\thanks{Email: p.e.hydon@kent.ac.uk}\\ \vskip2pt
SMSAS, University of Kent, Canterbury, CT2 7FS, UK\\\vskip6pt
{\sc and}\\ \vskip6pt 
{\sc L. Peng}\thanks{Email: l.peng@aoni.waseda.jp}\\  \vskip2pt 
Waseda Institute for Advanced Study, Waseda University, Tokyo, 169-8050, Japan
}

\shortauthorlist{E.L. Mansfield \emph{et al.}}

\maketitle

\begin{abstract}
{We consider the calculation of Euler--Lagrange systems of ordinary difference equations, including the difference Noether's Theorem, in the light of the recently-developed calculus of difference invariants and discrete moving frames. We introduce the difference moving frame, a natural discrete moving frame that is adapted to difference equations by prolongation conditions.

For any Lagrangian that is invariant under a Lie group action on the space of dependent variables, we show that the Euler--Lagrange equations can be calculated directly in terms of the invariants of the group action. Furthermore, Noether's conservation laws can be written in terms of a difference moving frame and the invariants. We show that this form of the laws can significantly ease the problem of solving the Euler--Lagrange equations, and
we also show how to use a difference frame to integrate Lie group invariant difference equations. In this Part I, we illustrate the theory by applications to Lagrangians invariant under various solvable Lie groups. The theory is also generalized to deal with variational symmetries that do not leave the Lagrangian invariant. 

Apart from the study of systems that are inherently discrete, one significant application is to obtain geometric (variational) integrators that have finite difference approximations of the continuous conservation laws embedded \textit{a priori}. This is achieved by taking an invariant finite difference Lagrangian in which the discrete invariants have  the correct continuum limit to their smooth counterparts.
We show the calculations for a discretization of the Lagrangian for Euler's elastica, and compare our discrete solution to that of its smooth continuum limit. }
{Noether's Theorem, Finite Difference, Discrete Moving Frames}
\end{abstract}
\tableofcontents

\section{Introduction}

Conservation laws are among the most fundamental attributes of a given system of partial differential equations. {{They constrain all solutions and have a topological interpretation as cohomology classes in the restriction of the variational bicomplex to solutions of the given system}} \citep{Vin}. Their importance has led to a major theme in geometric integration that seeks to construct finite difference approximations that preserve conservation laws in some sense. Approaches include methods that preserve symplectic or multisymplectic structures, energy, and other conservation laws. Variational integrators exploit the structure inherent in variational problems, where conservation laws are associated with symmetries.

Much of physics is governed by variational principles, with symmetries leading to conservation laws and Bianchi identities. Noether's first (and best-known) theorem relates an $R$-dimensional Lie group of symmetries of a given variational problem to $R$ linearly independent conservation laws of the system of Euler--Lagrange differential equations \citep{Noether}. A second theorem in her seminal paper deals with variational problems that have gauge symmetries, and there is now a bridging theorem that includes all intermediate cases \citep{HyMa}. {{For a translation into English of Noether's paper and an historical survey of the context and impact of Noether's theorems and subsequent generalizations, see \citet{KS}.}} Each of these theorems has now been adapted to difference equations \citep{Dorod,HyMa,Hydon}.

Noether's (first) Theorem may be used to derive conservation laws from a finite-dimensional Lie group of variational point symmetries that act on the space of independent and dependent variables. One can work in terms of the given variables, but for complex problems it is usually more efficient to factor out the Lie group action from the outset and work entirely in terms of invariants and the equivariant frame. Having solved a simplified problem for the invariants, one can then construct the solution to the original problem. This divide-and-conquer approach has recently been achieved, for differential equations, by using moving frame theory. These results were presented in \citet{GonMan} for all three inequivalent $SL(2)$ actions in the complex plane and in \citet{GonMan2} for the standard $SE(3)$ action.
Finally, in \citet{GonMan3} the calculations were extended to cases where the independent variables are not invariant under the group action, which is the case for many physically important models.

The theory and applications of Lie group based moving frames are now well established, and provide an invariant calculus to study differential systems that are either invariant or equivariant under the action of a Lie group.
Associated with the name of \'Elie \citet{Cartan},
who used \textit{rep\`eres mobile} to solve equivalence
problems in differential geometry, the ideas go back to earlier works, for example by \citet{cotton} and \citet{Darboux}. 

From the point of view of symbolic computation, a breakthrough
in the understanding of Cartan's methods for differential systems came in a series of papers
by \citet{FO2,FO1}, \citet{Omulti,Ojoint}, \citet{hubertAA,hubertAD,hubertAC} and \citet{hubertB,hubertA}, 
which provide a coherent, rigorous and constructive moving
frame method. The resulting differential invariant calculus is the subject of the textbook by \citet{Mansfield:2010aa}.
Applications include integration of 
Lie group invariant differential equations \citep{Mansfield:2010aa}, the calculus of variations
and Noether's Theorem
\citep[see for example][]{GonMan,GonMan2,kogan}, and integrable systems \citep[for example][]{ManKamp,M2,M1,M3}. 

Moving frame theory assumes that the Lie group acts on a continuous space. For spaces in which some variables are discrete, the theory must be modified.
The first results for the computation of discrete invariants using group-based moving frames were given by Olver (who called them 
joint invariants) in \citet{Ojoint};
modern applications to date include computer vision  \citep{Ogen} and numerical schemes
for systems with a Lie symmetry \citep{kimA,kimB,kimC,ManHy,RebVal}.
While moving frames for discrete applications as formulated
by Olver do give generating sets of discrete invariants, the recursion formulae for differential invariants
(which were so successful for the application
of moving frames to  calculus-based results) do not generalize well to joint invariants. In particular,
joint invariants do not seem to have computationally useful recursion formulae under the shift operator (which is defined below).
To overcome this  problem, \cite*{MBW} introduced the notion of a \textit{discrete moving frame}, which is essentially 
a sequence of frames. In that paper  discrete recursion formulae were proven for small computable generating sets of invariants,  called the \textit{discrete Maurer--Cartan invariants},
and  their \textit{syzygies} (that is,  their recursion relations) were investigated.  

Difference equations arise as models in their own right, not just as approximations to differential equations. Some have conservation laws; for instance, discrete integrable systems have infinite hierarchies of conservation laws \citep{MWX}. However, the geometry underlying finite difference conservation laws is much less well-understood than its counterpart for differential equations. Discrete moving frames are widely applicable to discrete spaces, but they do not incorporate the prolongation structure that is inherent in difference equations. This paper describes the necessary modification to incorporate this structure, difference moving frame theory, and applies it to ordinary difference equations with variational symmetries. This makes it possible to factor out the symmetries and write the Euler--Lagrange equations entirely in terms of invariant variables and the conservation laws  in terms of the invariants and the frame. 

We consider systems of ordinary difference equations (\odese) whose independent variable is $n\in\mathbb{Z}$, with dependent variables $\mathbf{u}=(u^1,\dots, u^q)\in \mathbb{R}^q$. A system of \odes is a given system of relations between the quantities $u^\alpha(n+j)$ for a finite set of integers $j$. The system holds for all $n$ in a given connected domain (interval), {{which may or may not be finite,}} so it is helpful to suppress $n$ and use the shorthand $u^\alpha_j$ for $u^\alpha(n+j)$ and $\mbu_j$ for $\mbu(n+j)$.

The (forward) shift operator $\es$ acts on functions of $n$ as follows:
\[
\es:n\mapsto n+1,\qquad \es:f(n)\mapsto f(n+1),
\]
for all functions $f$ whose domain includes $n$ and $n+1$. In particular, $$\es:u^\alpha_j\mapsto u^\alpha_{j+1}$$
on any domain where both of these quantities are defined.
The forward difference operator is $\es-\id$, where $\id$ is the identity operator:
\[
\id:n\mapsto n,\qquad \id:f(n)\mapsto f(n),\qquad \id:u^\alpha_j\mapsto u^\alpha_j.
\]

A variational system of \odes is obtained by extremizing a given functional,
$\mathcal{L}[\mbu]=\displaystyle{\sum_n}\, \upL(n,\mbu_0,\dots, \mbu_J)$, where the sum is taken over all $n$ in a given interval{{, which need not necessarily be bounded}}; the Lagrangian $\upL$ depends on only a finite number of arguments. The extrema are given by the condition
$$\displaystyle\frac{{\rm d}}{{\rm d}\epsilon}\Big\vert_{\epsilon=0} \sum_n \upL(n,\mbu_0+\epsilon \mbw_0,\dots, \mbu_J+\epsilon \mbw_J) =0$$
for all functions $\textbf{w}:\mathbb{Z}\rightarrow\mathbb{R}^q$. It is well known that the extrema satisfy the following system of Euler--Lagrange (difference) equations \citep{HyMan,Kup}:
\begin{equation}\label{basicEL} \eual(\upL):=\sum_{j=0}^J  \es_{-{j}}\!\left(\frac{\partial\:\! \upL}{\partial u^\alpha_j}\right)=0{{,\qquad\text{where}\ \es_{-j}=(\es^{-1})^j}}.\end{equation}
Each $\eual(L)$ depends only on $n$ and $\mbu_{-J},\dots,\mbu_J$, so the Euler--Lagrange equations are of order at most $2J$. In the following, we develop an invariantized version of
these equations, together with invariant conservation laws that stem from Noether's Theorem.

In Section \ref{prolspace}, the natural geometric setting for difference equations is discussed. Just as a given differential equation can be regarded as a subspace of an appropriate jet space, a given difference equation is a subspace of an appropriate difference prolongation space.

 Section \ref{introdiffcalc} is a brief review of the difference calculus of variations. The methods that we will develop emulate these calculations as far as possible, but using the invariant difference calculus. In Section \ref{intromovfram}, we give a short overview of continuous and discrete moving frames, and introduce the difference moving frame, which gives the geometric framework for our results. A running example is used from here on to illustrate how the theory is applied.
 
 In Section \ref{inveul}, we show how a difference moving frame can be used to calculate the difference Euler--Lagrange equations directly in terms of the invariants. This calculation yields boundary terms that can be transformed into the conservation laws, which require both invariants and the frame for their expression. Section \ref{moregrpsec} introduces the adjoint representation of the frame, enabling us (in Section \ref{Conlaws}) to state and prove key results on the difference conservation laws that arise via the difference analogue of Noether's Theorem.
 
Section \ref{SolveSec} shows how the difference moving frame may be used to integrate a difference system that is invariant under a Lie group action. Further, we show how the conservation laws and the frame together may be used to ease the integration process, whether or not one can solve for the frame.

In Section \ref{Divgen}, we generalize the difference frame version of Noether theory to include variational symmetries that do not leave the Lagrangian invariant. (Their counterparts for differential equations are sometimes called divergence symmetries.) Consequently, difference moving frames may be used to solve or simplify ordinary difference systems with any finite-dimensional Lie group of variational symmetries. This is illustrated in Section \ref{Exgen}.

The paper concludes with another use of difference moving frames: to create symmetry-preserving numerical approximations. Section \ref{Elastica} illustrates this for the Euler elastica, which is invariant under the Euclidean group action. Smooth Lagrangians which are invariant under the Euclidean action can be expressed in terms of the Euclidean curvature and arc length, and the Noether laws for these are the conservation of linear and angular momenta. {We demonstrate for this example that by taking a difference frame which converges, in some sense, to a smooth frame, then we obtain simultaneous convergence of the Lagrangian, the Euler--Lagrange system and all three conservation laws.   The specific difference Lagrangian we consider is a discrete analogue of that for Euler's elastica, and we show how our results compare with that of the smooth Euler--Lagrange equation, solved using the analogous  theory of smooth moving frames. In effect, we show how the design of the approximate Lagrangian can yield a discrete Euler--Lagrange system which is a variational integrator and which respects difference analogues of all three conservation laws.} 



\section{Difference prolongation spaces}\label{prolspace}

A given differential equation can be expressed geometrically as a {{variety}} in an appropriate jet space whose coordinates are the independent and dependent variables, together with sufficiently many derivatives of the dependent variables \citep{Opurp}. There is an analogous geometric structure for difference equations, but jet spaces and derivatives are replaced by difference prolongation spaces and shifts respectively. {{While equations may have singularities, techniques described in this paper are valid only away from these, and hence we do not consider such points here.}}

The difference prolongation spaces are obtained from the space of independent and dependent variables, $\mathbb{Z}\times\mathbb{R}^q$. Over each base point $n\in\mathbb{Z}$, the dependent variables take values in a continuous fibre $U\subset \mathbb{R}^q$, which has the coordinates $\mbu=(u^1,\dots,u^q)$. For simplicity, we shall assume that all structures on each fibre are identical; the necessary modifications when this does not hold are obvious but can be messy.

It is useful to regard $n$ as representing a given (arbitrary) base point and to prolong the fibre over $n$ to include the values of $\mbu$ on other fibres. For all sequences $\big(\mbu(m)\big)_{\!m\,\in\,\mathbb{Z}}$, let $\mbu_j$ denote $\mbu(n+j)$. Then the fibre over $n$ is $P^{(0,0)}_n(U)\simeq U$, with coordinates $\mbu_0$. The first forward prolongation space over $n$ is $P^{(0,1)}_n(U)\simeq U\times U$ with coordinates $z=(\mbu_0,\mbu_1)$. Similarly, the $J^{\mathrm{th}}$ forward prolongation space over $n$ is the product space $P^{(0,J)}_n(U)\simeq U\times\cdots\times U$ ($J+1$ copies) with coordinates $z=(\mbu_0,\mbu_1,\dots,\mbu_J)$. More generally, one can include both forward and backward shifts, obtaining the prolongation spaces
$P^{(J_0,J)}_n(U)\simeq U\times\cdots\times U$ ($J-J_{0}+1$ copies) with coordinates $z=(\mbu_{J_{0}},\dots,\mbu_J)$, where $J_{0}\leq 0$ and $J\geq 0$. 

Prolongation spaces are equipped with an ordering that enables sequences to be represented as points in the appropriate prolongation space. Consequently, a given difference equation, $\mathcal{A}(n,\mbu_{J_{0}},\dots,\mbu_J)=0$, corresponds to a variety in any prolongation space that contains the submanifold $P^{(J_0,J)}_n(U)$. This makes it possible to apply continuous methods to difference equations  {{(locally, away from singularities)}}.

Every prolongation space $P^{(J_0,J)}_n(U)$ is a {{submanifold}} of the \textit{total prolongation space over $n$}, $P^{(-\infty,\infty)}_n(U)$ with coordinates $z=(\dots,\mbu_{-2},\mbu_{-1},\mbu_0,\mbu_1,\mbu_2,\dots)$. As $n$ is a free variable, the same structures are repeated over each $n$. This yields the natural map
\[
\pi:P^{(-\infty,\infty)}_n(U)\longrightarrow P^{(-\infty,\infty)}_{n+1}(U),\qquad \pi:z\mapsto \hat{z};
\]
here the coordinates on $P^{(-\infty,\infty)}_{n+1}(U)$ are distinguished by a caret, so $\widehat{\mbu}_j$ denotes $\mbu(n+1+j)$ {{for all sequences $\big(\mbu(m)\big)_{\!m\,\in\,\mathbb{Z}}$. Consequently, each $\widehat{\mbu}_j$ is represented in $P^{(-\infty,\infty)}_n(U)$ by $\mbu_{j+1}=\es\mbu_j$, where we regard the shift $\es$ as an operator on $P^{(-\infty,\infty)}_n(U)$. Similarly, variables on any prolongation space over a point $m$ may be represented as equivalent variables in $P^{(-\infty,\infty)}_n(U)$ by applying the $(m-n)^{\mathrm{th}}$ power of the shift operator $\es$. We denote the $j^{\mathrm{th}}$ power of $\es$ by $\es_j$, so that $\mbu_j=\es_j\mbu_0$ for each $j\in\mathbb{Z}$.}}

For \odese, it is enough to use the restriction of $\es$ to finite prolongation spaces. To accommodate difference equations on a finite or semi-infinite interval, we add the constraint that $\mbu_{j+1}=\es \mbu_j$ is defined only if $n+j$ and $n+1+j$ are in the interval.

In the remainder of the paper, we treat $n$ as fixed, using powers of the shift operator $\es$ to represent structures on prolongation spaces over any base point $m$ as equivalent structures on all sufficiently large prolongation spaces over $n$. In \S\ref{intromovfram}, we will use this property to construct difference moving frames.

{Throughout, we work  formally, without considering convergence of sums or integrals.}

\section{The difference variational calculus}\label{introdiffcalc}

Consider a functional of the form
\begin{equation}\label{basicLag}\mathcal{L}[\mbu]=\sum \upL(n,\mbu_0,\mbu_1, \dots, \mbu_J),\end{equation}
where $\mbu_j=(u_j^1,\dots, u_j^q)\in \mathbb{R}^q$. Here and henceforth, the unadorned summation symbol denotes summation over $n$; the range of this summation is a given interval in $\mathbb{Z}$, which can be unbounded.  For sums over all other variables, we will use the Einstein summation convention as far as possible, to avoid a proliferation of summation symbols.
The \textit{variation} of $\mathcal{L}[\mbu]$ in the direction $\mbw$ is
\begin{equation}\label{varL}
\frac{{\rm d}}{{\rm d} \epsilon}\Big\vert_{\epsilon=0} \mathcal{L}[\mbu+\epsilon \mbw]=\sum  w^{\alpha}_j\,\displaystyle\frac{\partial\:\! \upL}{\partial u^{\alpha}_j }\,.
\end{equation}
Summing by parts, using the identity
\begin{equation}\label{sumpartdef} (\es_jf)\,g = f\, \es_{-j}g+\left(\es_j-\id\right)(f\, \es_{-j}g)\end{equation}
and pulling out the factor $(\es -\id)$ from $(\es_j-\id)$, we obtain
\begin{equation}\label{firstVar}
w^{\alpha}_j\,\displaystyle\frac{\partial\:\! \upL}{\partial u^{\alpha}_j }=
w^{\alpha}_0\,\es_{-{j}} \displaystyle\frac{\partial\:\! \upL}{\partial u^{\alpha}_{j} }+(\es -\id)A_\mbu(n,\mbw),
\end{equation}
where
\[
A_\mbu(n,\mbw)=\sum_{j=1}^J\sum_{l=0}^{j-1}\,\es_{\,l}\!\left\{
w^{\alpha}_0\,\es_{-{j}} \displaystyle\frac{\partial\:\! \upL}{\partial u^{\alpha}_{j} }\right\}.
\]
The sum over $n$ of the differences $(\es -\id)A_\mbu$ telescopes, contributing only boundary terms to the variation. So for all variations to be zero,  $\mbu$ must solve the Euler--Lagrange system of difference equations
\begin{equation}\label{ELeqnOne}
\eual(\upL) := \es_{-{j}} \frac{\partial\:\! \upL}{\partial u^{\alpha}_{{j}} }=0,\qquad  \alpha=1,\dots, q.
\end{equation}
Moreover, the boundary terms yield natural boundary conditions that must be satisfied if $\mbu$ is not fully constrained at the boundary.

\begin{example}\label{simpleExOne}    
As a running example, we will consider a functional with two dependent variables, $\mbu=(x,u)$, that is of the form 
$$\mathcal{L}[x,u] = \sum \upL(x_0,u_0,x_{1},u_{1},u_{2}). $$  
Setting $\mbw=(w^x, w^u)$, the variation is
$$\frac{{\rm d}}{{\rm d}\epsilon}\Big\vert_{\epsilon=0} \mathcal{L}[x+\epsilon w^x,u+\epsilon w^u]=\sum \big\{ w^x_0\,\eux(\upL) +w^u_0\,\euu(\upL) + (\es -\id)A_\mbu(n,\mbw)\big\}.$$
There are two Euler--Lagrange equations, one for each dependent variable:
$$ \eux(\upL):=\frac{\partial\:\! \upL}{\partial x_0}+\es_{-1}\frac{\partial\:\! \upL}{\partial x_{1}}\,=0,\qquad\euu(\upL):=\frac{\partial\:\! \upL}{\partial u_0}+\es_{-1}\frac{\partial\:\! \upL}{\partial u_{1}}
+\es_{-2}\frac{\partial\:\! \upL}{\partial u_{2}}=0.$$
The remaining terms come from $A_\mbu(n,\mbw)=A^x+A^u$, where 
$$A^x=w^x_0\,\es_{-1}\frac{\partial\:\! \upL}{\partial x_{1}}\,,\qquad A^u=w^u_0\,\es_{-1}\frac{\partial\:\! \upL}{\partial u_{1}} + \left(\es+\id\right)\left(w^u_0\,\es_{-2}\frac{\partial\:\! \upL}{\partial u_{2}}\,\right).
$$
\end{example}

So far, we have considered all variations $\mbw$, without reference to the Lagrangian. However, variations $\bphi$ that do not change the functional $\mathcal{L}[\mbu]$ are of special importance. They leave the Lagrangian $\upL$ invariant, up to a total difference term.

\begin{definition}\label{infpart1}
Suppose that a non-zero function $\bphi=(\phi^1(n,\mbu),\dots,\phi^q(n,\mbu))^T$ satisfies
\begin{equation}\label{invLagOne}\phi^{\alpha}_{j}(n,\mbu)\, \frac{\partial\:\! \upL}{\partial u^{\alpha}_{j} } = (\es -\id )B(n,\mbu),\qquad\text{where}\quad\phi^\alpha_j=\es_j\phi^\alpha_0,\end{equation}
for some $B(n,\mbu)$ (which may be zero).
Then the Lagrangian $\upL$ is said to have a \emph{variational symmetry}\footnote{{This is shorthand for a one-parameter local Lie group of variational symmetries, see \citet{Obook}.}} \emph{with characteristic} $\bphi$. The Lagrangian is invariant under this symmetry if $B=0$. \end{definition}

The relationship between variational symmetries and characteristics will be made clear in \S \ref{moregrpsec}. The next theorem explains why these symmetries are important.
\begin{theorem}[Difference Noether's Theorem]\label{Noether1} Suppose that a Lagrangian $\upL$ has a variational symmetry with characteristic $\bphi\ne \mathbf{0}$.
If $\mbu=\bar{\mbu}$ is a solution of the Euler--Lagrange system for $\upL$ then
\begin{equation} \label{conLawBA}\big(\,(\es-\id)\! \left\{A_\mbu(n,\bphi)-B(n,\mbu)\right\}\big)\big\vert_{\mbu=\bar{\mbu}}=0.
\end{equation} 
\end{theorem}
\begin{proof} Substituting $\bphi$ for $\mbw$ in (\ref{firstVar}) gives
$$\phi^{\alpha}(n,\mbu)\eual(\upL) + (\es-\id) A_\mbu(n,\bphi)=
\phi^{\alpha}_{j}(n,\mbu)\, \frac{\partial\:\! \upL}{\partial u^{\alpha}_{j} } = (\es -\id )B(n,\mbu).$$
The result follows immediately.
\end{proof}

The expression in Equation (\ref{conLawBA}) is a {\em conservation law} for the Euler--Lagrange system. As there is only one independent variable, the expression in braces is a first integral, so every solution of the Euler--Lagrange system satisfies
\[
\big\{A_\mbu(n,\bphi)-B(n,\mbu)\big\}\big\vert_{\mbu=\bar{\mbu}}=c,
\]
where $c$ is a constant.

\begin{table}
	\centering{
		\begin{tabular}{|p{0.5cm}|p{0.4cm}| p{9cm} |}
			\hline
			\vspace*{-6pt}$\phi^x$&\vspace*{-6pt}$\phi^u$&\vspace*{-6pt}First integral: $A_\mbu(n,\bphi)$\\[2pt]
			\hline 
			$3x$&$u$&$\phantom{\Bigg\vert}3 x_{0}\,\es_{-1}\ds\frac{\partial\:\! \upL}{\partial x_{1}} + u_{0}\,\es_{-1}\ds\frac{\partial\:\! \upL}{\partial u_{1}}$
			$+(\es+\id)\left(u_0\,\es_{-2}\ds\frac{\partial\:\! \upL}{\partial u_{2}}\right)$\\[15pt]
			\hline
			1&0&$\phantom{\Bigg\vert}\es_{-1}\ds\frac{\partial\:\! \upL}{\partial x_{1}}$
			\\ [12pt]\hline
			0&1&$\phantom{\Bigg\vert}\es_{-1}\ds\frac{\partial\:\! \upL}{\partial u_{1}} + (\es+\id) \es_{-2}\ds\frac{\partial\:\! \upL}{\partial u_{2}}\equiv\es_{-1}\ds\frac{\partial\:\! \upL}{\partial u_{2}}-\ds\frac{\partial\:\! \upL}{\partial u_{0}}$\\ \hline
		\end{tabular}
		\caption{Characteristics and first integrals for the Lagrangian (\ref{babylagrangian}).}\label{tableEG}}
\end{table}
\smallskip

\begin{example}(Example \ref{simpleExOne} cont.)\  
For instance, the Lagrangian
\begin{equation}\label{babylagrangian}
\upL(x_0,u_0,x_{1},u_{1},u_{2})= \frac{x_1-x_0}{\{(u_2-u_1)(u_1-u_0)\}^{3/2}}
\end{equation}
has three variational symmetries, all with $B=0$. Table \ref{tableEG} lists the corresponding first integrals for every Lagrangian $\upL(x_0,u_0,x_{1},u_{1},u_{2})$ that has these symmetries.

We will see in the sequel that the first symmetry arises from the invariance of the Lagrangian under the scalings $(x,u)\mapsto (\lambda^3 x, \lambda u)$, for $\lambda\in \mathbb{R}^{+}$, the second arises from invariance under translations in $x$, that is, $x\mapsto x+a$ for all $a\in\mathbb{R}$, and the third arises from invariance under translations in $u$, namely $u\mapsto u+b,\ b\in\mathbb{R}$.

Despite having three first integrals for the system of Euler--Lagrange equations, the expressions are unwieldy for the given Lagrangian \eqref{babylagrangian} and the system remains difficult to solve. We will show that the necessary insight into the solution set is obtained by using coordinates that are adapted to the three symmetries.
So for the rest of the running example, we will consider only those Lagrangians $\upL(x_0,u_0,x_{1},u_{1},u_{2})$ that have the same three symmetries. A major advantage of using symmetry-adapted coordinates is that one can deal with all such Lagrangians together.
\end{example}

To obtain adapted coordinates for general variational problems, we need the machinery of moving frames and, in particular, difference moving frames.

\section{Moving frames}\label{intromovfram}

This section outlines the basic theory of the moving frame and its extension to the discrete moving frame on an appropriate prolongation space, as illustrated by the running example. The discrete moving frame is essentially a sequence of moving frames. We introduce the \emph{difference moving frame}, which is equivalent to a discrete moving frame (on a prolongation space) that is subject to further prolongation conditions. It is analogous to the continuous moving frame on
a given jet space, adapted to difference equations that are invariant or equivariant under a Lie group action.

\subsection{Lie group actions, invariants and moving frames}

Given a Lie group $G$ that acts on a manifold $M$, the group action of $G$ on $M$ is a map 
\begin{equation}
G\times M \rightarrow M, \quad (g,z)\mapsto g\cdot z.
\end{equation}
For a left action, 
$$
g_1\cdot(g_2\cdot z) = (g_1g_2)\cdot z;
$$  
similarly, for a right action
$$
g_1\cdot (g_2\cdot z)=(g_2g_1)\cdot z.
$$
Given a left action $(g,z)\mapsto g\cdot z$, it follows that $(g,z)\mapsto g^{-1}\cdot z$  is a right action.
In practice both right and left actions occur, and the ease of the calculations can differ considerably, depending on the choice. In a theoretical development, however,  only one is needed, so we restrict ourselves in the following to left actions.
\smallskip 

\noindent\textbf{Remark on notation.} For ease of exposition, a tilde will be used to denote a transformed variable; for instance, $g\cdot z=\tilde{z}$. In this notation, $g$ is implicit. \smallskip

\begin{figure}[t]
	\begin{center}
		\includegraphics[trim=0 340 0 0,clip,scale=0.65]{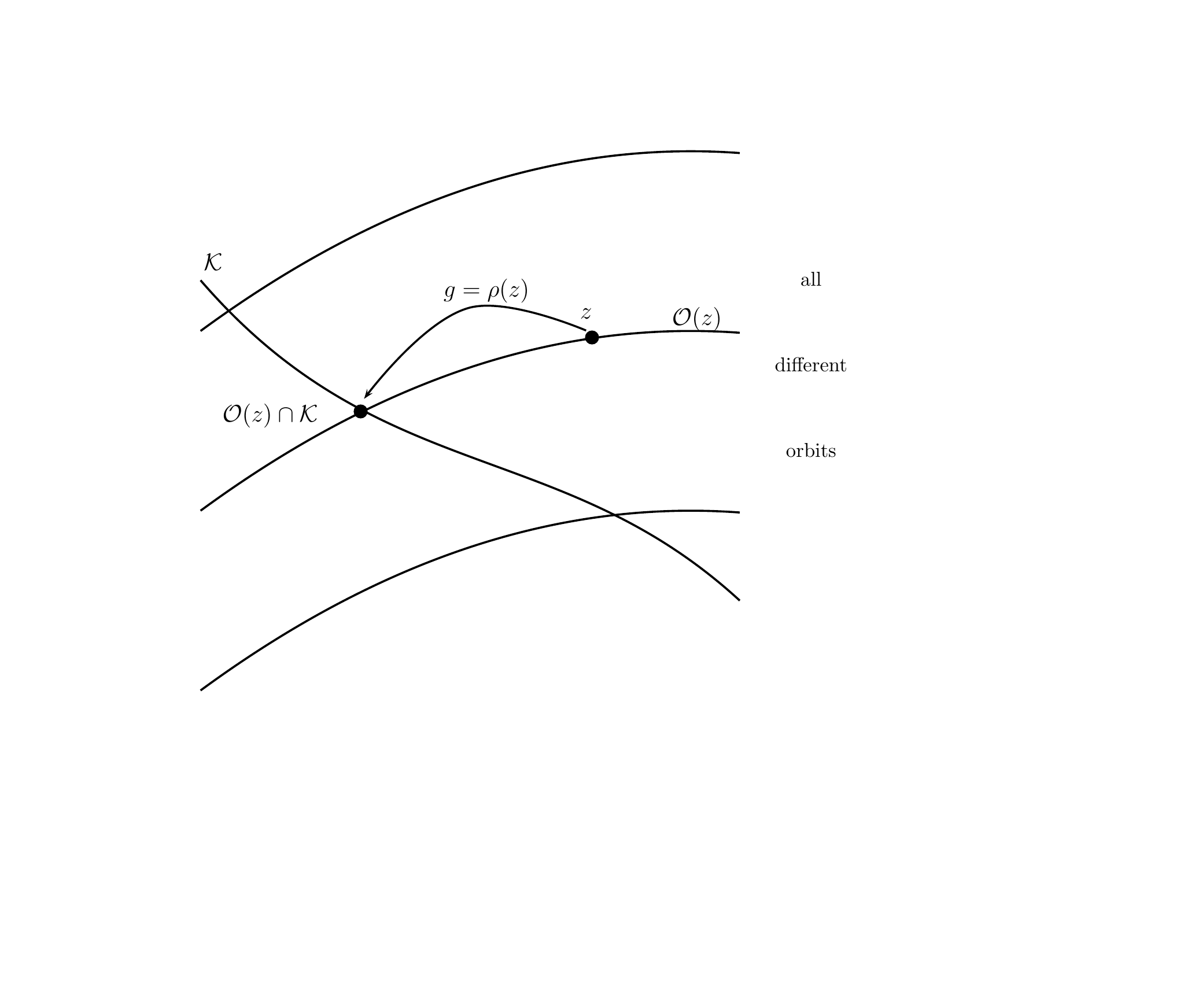}
		\caption{A moving frame defined by a cross-section}
		\label{fig:movingframe}
	\end{center}
\end{figure}

\noindent\textbf{Important assumption.}
In the following outline of moving frame theory, the action is assumed to be free and regular on $M$; see \citet{Mansfield:2010aa} for details. (If it is not, replace $M$ by a domain $\man$ on which the action is free and regular.) Consequently, there exists a cross section $\mathcal{K}\subset M$ that is transverse to the orbits $\mathcal{O}(z)$ and, for each $z\in M$, the set $\mathcal{K}\cap \mathcal{O}(z)$ has just one element, the projection of $z$ onto $\mathcal{K}$, as shown in Figure \ref{fig:movingframe}.

Using the  cross-section $\mathcal{K}$, a moving frame for the group action on {{a neighbourhood $\mathcal{U}\subset M$ of $z$}} can be defined as follows.
\begin{definition}[Moving Frame]
Given a smooth Lie group action $G\times M\rightarrow M$, a moving frame is an equivariant map $\rho: \mathcal{U}\subset M\rightarrow G$. Here $\mathcal{U}$ is called the domain of the frame.
\end{definition}
A left equivariant map satisfies $\rho(g\cdot z) = g\rho(z)$, and a right equivariant map satisfies $\rho(g\cdot z)=\rho(z)g^{-1}$. The frame is called left or right accordingly. 

In order to find the frame, let the cross-section $\mathcal{K}$ be given by a system of equations $\psi_r(z)=0$, for $r=1,\ldots,R,$ where $R$ is the dimension of the group $G$. One then solves the
so-called normalization equations,
\begin{equation}\label{frameEqA}
\psi_r(g\cdot z) = 0, \qquad r=1,\ldots, R,
\end{equation}
for $g$ as a function of $z$.
The solution is the group element $g=\rho(z)$ that maps $z$ to its projection on $\mathcal{K}$ (see Figure \ref{fig:movingframe}).
In other words, the frame $\rho$ satisfies
$$
\psi_r(\rho(z)\cdot z)=0, \qquad r=1,\ldots, R. 
$$
The conditions on the action above are those for the Implicit Function Theorem to hold \citet{Hirsch}, so the solution $\rho$ is unique. A consequence of uniqueness is that
$$
\rho(g\cdot z)=\rho(z) g^{-1},
$$
that is, the frame is  \textit{right equivariant}, as both $\rho(g\cdot z) $ and $ \rho(z) g^{-1}$ solve the equation $\psi_r\left(\rho(g\cdot z)\cdot \left(g\cdot z\right)\right)=0$.
A left frame, which satisfies $ \rho(g\cdot z)=g\rho(z)$, is obtained by taking the inverse of a right frame. In practice, the ease of calculation can differ considerably depending on the choice of parity.

\begin{example}(Example \ref{simpleExOne} cont.)\  
It is straightforward to check that the Lagrangian \eqref{babylagrangian} is invariant under the scaling and translation group action on $\mathbb{R}^2$ given by
$$ (x,u)\mapsto (\lambda^3 x +a , \lambda u +b),\qquad \lambda\in\mathbb{R}^+,\  a,b\in\mathbb{R};$$
the Lie group is the semi-direct product, $\mathbb{R}^+\ltimes\mathbb{R}^2$.
The action is not free on the space $\mathbb{R}^2$ over $n$, which has coordinates $(x_0,u_0)$. To achieve freeness,  one needs to work in a higher-dimensional continuous space. For instance, the action is free on the first forward prolongation space $P_n^{(0,1)}(\mathbb{R}^2)$, which has coordinates $(x_0,u_0,x_1,u_1)$. On this prolongation space, the action is given by
$$ \left(x_0,u_0,x_1,u_1\right)\mapsto \left(\lambda^3 x_0 +a , \lambda u_0 +b,\lambda^3 x_1 + a, \lambda u_1+b\right).$$
If we choose the normalization equations to be 
$g\cdot (x_0, u_0,x_1,u_1)=(0,0,*,1)$, where $*$ is unspecified, the values of the parameters for the frame group element are\footnote{{As $\lambda>0$, this is valid throughout the half-space $\mathcal{U}=\{(x_0, u_0,x_1,u_1)\in P_n^{(0,1)}(\mathbb{R}^2):u_1>u_0\}$. For the half-space $u_1<u_0$, the normalization $g\cdot (x_0, u_0,x_1,u_1)=(0,0,*,-1)$ would be appropriate.}}
$$\lambda=\frac{1}{u_1-u_0}\,, \qquad a=-\,\frac{x_0}{(u_1-u_0)^3}\,,\qquad b=-\,\frac{u_0}{u_1-u_0}\,.$$
A standard matrix representation for the generic group element is 
$$g(\lambda,a,b) =  \left(\begin{array}{ccc} \lambda^3 &0&a\\ 0&\lambda & b\\ 0&0&1\end{array}\right),\qquad\text{with}\quad \left(\begin{array}{c} g\cdot x\\ g\cdot u\\ 1\end{array}\right)=g(\lambda,a,b)\left(\begin{array}{c} x\\ u\\ 1\end{array}\right).$$
With this representation, which is faithful, the frame is 
$$\rho(x_0,u_0,x_1,u_1) =\left(\begin{array}{ccc}\ds\frac1{(u_1-u_0)^3} &0&-\,\ds\frac{x_0}{(u_1-u_0)^3}\\ 0&\ds\frac1{u_1-u_0} &  -\,\ds\frac{u_0}{u_1-u_0}\\ 0&0&1\end{array}\right).$$
The equivariance  is easily checked in matrix form: for $z=(x_0,u_0,x_1,u_1)$, 
$$\begin{array}{rcl}\rho(g\cdot z)
&=& \left(\begin{array}{ccc}\ds\frac1{\lambda^3(u_1-u_0)^3} &0&-\,\ds\frac{x_0+a/\lambda^3}{(u_1-u_0)^3}\\[10pt] 0&\ds\frac1{\lambda(u_1-u_0)} &  -\,\ds\frac{u_0+b/\lambda}{u_1-u_0}\\[10pt] 0&0&1\end{array}\right)\\[12pt]
&=&\left(\begin{array}{ccc}\ds\frac1{(u_1-u_0)^3} &0&-\,\ds\frac{x_0}{(u_1-u_0)^3}\\[12pt] 0&\ds\frac1{u_1-u_0} &  -\,\ds\frac{u_0}{u_1-u_0}\\[10pt] 0&0&1\end{array}\right)
\left(\begin{array}{ccc} \lambda^{-3}&0&-\,\ds\frac{a}{\lambda^3}\\[10pt] 0&\lambda^{-1} & -\,\ds\frac{b}{\lambda}\\[10pt] 0&0&1\end{array}\right)\\[12pt]
&=& \rho(z)  g(\lambda,a,b)^{-1}.\end{array} $$
\end{example}

Returning to the general theory, the requirement that frames are equivariant enables one to obtain invariants of the group action.

\begin{lemma}[Normalized Invariants]
Given a left or right action $G\times M \rightarrow M$ and a right frame $\rho$, then
 $\iup(z)=\rho(z)\cdot z$, for $z$ in the domain of the frame $\rho$, is invariant under the group action.
\end{lemma}
\begin{proof}First apply the group action to $z$; then, by definition,
\begin{equation}\label{proofIzInv}
\iup(g\cdot z)=\rho(g\cdot z)\cdot(g\cdot z) = \rho(z)\cdot g^{-1}g\cdot z =\rho(z)\cdot z = \iup(z),
\end{equation}
so $\iup(z)$ is an invariant function. \end{proof}

\smallskip
\begin{definition} The normalized invariants are the components of $\iup(z)$.\end{definition}

{{\begin{definition}\label{DefnGenSet} A set of invariants  is said to be a \emph{generating}, or \emph{complete}, set for an algebra of invariants if any invariant in the algebra can be written as a function of elements of the generating set.\end{definition}}}

We now state the Replacement Rule (see \citet{FO2}), from which it follows that the normalized invariants provide a set of generators for the algebra of invariants.

\begin{theorem}[Replacement Rule]\label{reprule}
If $F(z)$ is an invariant of the given action $G\times M\rightarrow M$
for a right moving frame $\rho$ on $M$ then $F(z)=F(\iup(z))$.
\end{theorem}
\begin{proof}
As $F(z)$ is invariant, $F(z)=F(g\cdot z)$ for all $g\in G$. Setting $g=\rho(z)$ and using the definition of $\iup(z)$ yields the required result.
\end{proof}

\begin{definition}[Invariantization Operator]\label{invOpDef} Given a right moving frame $\rho$, the map $z\mapsto \iup(z)=\rho(z)\cdot z$ is called the \emph{invariantization operator}. This operator extends to functions as
$f(z)\mapsto f(\iup(z))$, and $f(\iup(z))$ is called the \emph{invariantization} of $f$.\end{definition}

\noindent If $z$ has components $z^\alpha$, let $\iup(z^\alpha)$ denote the $\alpha^{\mathrm{th}}$ component of $\iup(z)$. 

\begin{example}(Example \ref{simpleExOne} cont.)\  
The action of the frame on $z=(x_0,u_0,x_1,u_1)\in P_n^{(0,1)}(\mathbb{R}^2)$ is
$$\rho(z) \cdot z  = \left(0,0,\ds\frac{x_1-x_0}{(u_1-u_0)^3}\,,1\right)$$
which is seen to be an invariant of the group action (constants are always invariants). With this same frame, higher forward prolongations of $\mathbb{R}^2$ yield more invariants.
For $z=\left( x_0,u_0,x_1,u_1, \dots, x_J,u_J\right)\in P_n^{(0,J)}(\mathbb{R}^2)$,
$$ \rho(z)\cdot z= \left(0,0, \ds\frac{x_1-x_0}{(u_1-u_0)^3}\,,1, \ds\frac{x_2-x_0}{(u_1-u_0)^3}\,,\ds\frac{u_2-u_0}{u_1-u_0}\,, \dots ,
\ds\frac{x_J-x_0}{(u_1-u_0)^3}\,,\ds\frac{u_J-u_0}{u_1-u_0}
\right).$$
Difference Euler--Lagrange equations typically involve both positive and negative indices $j$, so in general, one must prolong forwards and backwards. For instance, for elements $z=(x_{-1},u_{-1},x_0,u_0,x_1,u_1)\in P_n^{(-1,1)}(\mathbb{R}^2)$,
$$ \rho(z)\cdot z= \left(\ds\frac{x_{-1}-x_0}{(u_1-u_0)^3}\,,\ds\frac{u_{-1}-u_0}{u_1-u_0}\,,0,0, \ds\frac{x_1-x_0}{(u_1-u_0)^3}\,,1
\right),$$
and so on. These results are summarized by $$ \iup(x_j):=\rho\cdot x_j  = \frac{x_j-x_0}{(u_1-u_0)^3}\,,\qquad \iup(u_j):=\rho\cdot u_j =\frac{u_j-u_0}{u_1-u_0}\,,\qquad j\in\mathbb{Z}.$$
We now calculate recurrence relations for these invariants and show that all of them can be written in terms of two fundamental invariants,
 \begin{equation}\label{kappaandeta}\kappa =\iup(u_2)= \rho\cdot u_2,\qquad \eta=\iup(x_1)=\rho\cdot x_1, \end{equation}
  and their shifts.
For each $j\in\mathbb{Z}$,
$$ \es\big\{\iup(u_j)\big\}= \frac{u_{j+1}-u_1}{u_2-u_1}=\frac{\displaystyle\frac{u_{j+1}-u_0}{u_1-u_0}-\displaystyle\frac{u_1-u_0}{u_1-u_0}}{\displaystyle\frac{u_2-u_0}{u_1-u_0}-\displaystyle\frac{u_1-u_0}{u_1-u_0}}=\frac{\iup(u_{j+1})-1}{\kappa-1}\,.$$
We will later prove that each shift of an invariant is an invariant. So the last equality can be obtained more easily by using the Replacement Rule (Theorem \ref{reprule}):
\[
\es\big\{\iup(u_j)\big\}=\frac{u_{j+1}-u_1}{u_2-u_1}=\iup\left(\frac{u_{j+1}-u_1}{u_2-u_1}\right)=\frac{\iup(u_{j+1})-\iup(u_1)}{\iup(u_2)-\iup(u_1)}=\frac{\iup(u_{j+1})-1}{\kappa-1}\,.
\]
The calculation for $\es\{\iup(x_j)\}$ is similar. Together, these identities amount to
 \begin{equation}\label{examplebabyframeA:A} 
\iup(u_{j+1})=(\kappa -1)\,\es\big\{\iup(u_j)\big\} +1,\qquad \iup(x_{j+1})=(\kappa -1)^3\,\es\big\{\iup(x_j)\big\} +\eta.\end{equation}
This shows that the invariants with positive $j$ can be written in terms of $\kappa,\eta$ and their forward shifts. For convenience, let $\kappa_k=\es_k\kappa$ and $\eta_k=\es_k\eta$ for all $k\in\mathbb{Z}$.  
To find the invariants with negative indices $j$, invert \eqref{examplebabyframeA:A}:
$$\iup(u_{j-1})=\frac{\es_{-1}\{\iup(u_j)\}-1}{\kappa_{-1}-1},\qquad \iup(x_{j-1})=\frac{\es_{-1}\{\iup(x_j)\}-\eta_{-1}}{(\kappa_{-1}-1)^3}.$$
For instance,
$$ \iup(u_{-1})=\frac{-\,1}{\kappa_{-1}-1}\,,\qquad \iup(u_{-2})=\frac{-\,\kappa_{-2}}{(\kappa_{-2}-1)(\kappa_{-1}-1)}\,,$$
and so on. Note that all $\iup(u_{j})$, $\iup(x_{j})$ can be written in terms of $\eta$ and $\kappa$ and their shifts.

The invariants $\iup(x_j)$, $\iup(u_j)$ do not behave well under the
shift map, in the sense that $\es\{\iup(x_j)\}\ne \iup(x_{j+1})$ and $\es\{\iup(u_j)\}\ne \iup(u_{j+1})$.
Even though the shift map takes invariants to invariants,  writing shifts of invariantized variables in terms of shifts of $\eta$ and $\kappa$
involves complicated expressions. The discrete moving frame, defined next, will lead to the proper geometric setting that explains the origin of these expressions.
\end{example}

\subsection{Discrete moving frames}
A discrete moving frame is an analogue of a moving frame that is adapted to discrete base points. It amounts to a sequence of frames defined on a product manifold. More detail on discrete moving frames and their applications can be found in \citet{MBW} and \citet{GMBMAN}.

In this subsection, the manifold on which $G$ acts will be the Cartesian product manifold $\man  = M^N$. We assume that the action on $\man $ is free, taking the number of copies $N$ of the manifold ${M}$ to be as high as necessary.  This happens, for example, when the action is (locally) effective on subsets, see \citet{boutin} for a discussion of this and related issues; further, see \citet{Ojoint} for a pathological example where the product action is not free for any $N$.
Questions like the regularity and freeness of the action will refer to the diagonal action on the product; given the action $(g,z_j)\mapsto g\cdot z_j$ for $z_j\in { M}$, the diagonal action of $G$ on $z=(z_1,z_2,\dots, z_N)\in\man$ is
$$ g\cdot (z_1,z_2,\dots, z_N)\mapsto (g\cdot z_1, g\cdot z_2, \dots, g\cdot z_N).$$
\textbf{Important note}: \emph{Throughout this subsection, no assumptions are made about any relationship between the elements} $z_1,\dots,z_N$.

\begin{definition}[Discrete Moving Frames: \citet{MBW,GMBMAN}]
Let $G^N$ denote the Cartesian product of $N$ copies of the group $G$. 
A map
$$
\rho: M^N \rightarrow G^N,\qquad \rho(z)=(\rho_1(z),\dots,\rho_N(z))
$$
is a right discrete moving frame if 
$$
  \rho_k(g\cdot z)=\rho_k(z)g^{-1},\qquad k=1,\dots,N,
$$
and a left discrete moving frame if
$$\rho_k(g\cdot z)=g\rho_k(z),\qquad k=1,\dots,N.$$ 
\end{definition}
Obtaining a frame via the use of normalization equations yields a right frame. As the theory for right and left frames is parallel,   we  restrict ourselves to studying  right frames only. It is advisable when calculating examples, however, to check the parity of actions and frames, see \citet{Mansfield:2010aa} for the subtleties involved.

A discrete moving frame is a sequence of moving frames $(\rho_k)$ with a nontrivial intersection of domains which, locally, are uniquely determined by the cross-section $\mathcal{K}=(\mathcal{K}_1,\dots,\mathcal{K}_N)$ to the group orbit through $z$. The right moving frame component $\rho_k$ is the unique element of the group $G$ that takes $z$ to the cross section $\mathcal{K}_k$.
We also define for a right frame, the invariants
\begin{equation}\label{invariant}
I_{k,j}:= \rho_k(z)\cdot z_j.
\end{equation}
If $M$ is $q$-dimensional, so that $z_j$ has components $z_j^1,\dots,z_j^q$, the $q$ components of $I_{k,j}$ are the invariants
\begin{equation}
I^\alpha_{k,j}:= \rho_k(z)\cdot z^\alpha_j,\qquad \alpha=1,\dots q.
\end{equation}
These are  invariant by the same reasoning as for Equation (\ref{proofIzInv}).
For later use, let $\iup_k$ denote the invariantization operator with respect to the frame $\rho_k(z)$,
so that $$I_{k,j}=\iup_k(z_j),\qquad I_{k,j}^\alpha=\iup_k(z_j^\alpha).$$

\subsection{Difference moving frames}

The discrete moving frame is a powerful construction that can be adapted to any discrete domain. Typically, $\man$ represents the fibres $M$ over a sequence of $N$ discrete points. The geometric context may determine additional structures on $\man$.

From \S\ref{prolspace}, the fact that $n$ is a free variable allows us to replicate the same structures over each base point $m$, using powers of the natural map $\pi$. The shift operator enables these structures to be represented on prolongation spaces over any given $n$. This suggests that the natural moving frame for a given \ode has $\man=P_n^{(J_0,J)}(U)$ for some appropriate $J_0\leq 0$ and $J\geq 0$. Consequently, $N=J-J_0+1$; from here on, we replace the indices $1,\dots, N$ by $J_{0},\dots, J$. 

{{We now use $\mathcal{K}_k$ and $\rho_k$ to denote the cross-sections and frames on $\man$, respectively.}}
The cross-section over $n$, denoted $\mathcal{K}_0$, is replicated for all other base points $n+k$ if and only if the cross-section over $n+k$ is represented on $\man$ by
\begin{equation}\label{prolK}
\mathcal{K}_k=\es_k\mathcal{K}_0
\end{equation}
for all $k$. When this condition holds, $\rho_k=\es_k\rho_0$ (by definition) for all $k$; consequently, $\mathcal{K}_{k+1}=\es\mathcal{K}_k$ and $\rho_{k+1}=\es\rho_k$.

\begin{definition}
	A \emph{difference moving frame} is a discrete moving frame such that $\man$ is a prolongation space $P_n^{(J_0,J)}(U)$ and \eqref{prolK} holds for all $J_0\leq k\leq J$. 
\end{definition}
By definition, the invariants $I_{k,j}$  given by a difference moving frame satisfy
\begin{equation}\label{shiftoninvs} \es(I_{k,j})=I_{k+1,j+1},
\end{equation}
so every invariant $I_{k,j}$ can be expressed as a shift of $I_{0,j-k}$.

\begin{definition}[Discrete Maurer--Cartan invariants] \label{KmatDef} Given a right discrete moving frame $\rho$, the \emph{right discrete  Maurer--Cartan group elements} are
\begin{equation}\label{Kk}
K_k=\rho_{k+1}\rho_k^{-1}
\end{equation}
{{for $J_0\le k\le J-1$.}}
\end{definition}
As the frame is equivariant, each $K_k$ is invariant under the action of $G$. We call the components of the Maurer--Cartan elements the \emph{Maurer--Cartan invariants}. 

As $\rho_k$ is a frame for each $k$, the components of $\rho_k(z)\cdot z$ generate the set of all invariants by the Replacement Rule (Theorem \ref{reprule}). 
{{Moreover, for the two different frames $\rho_{k+1}$ and $\rho_k$, and for any invariant $F(z)$, the Replacement Rule gives
\[ F(z)=F(\rho_k(z)\cdot z)=F(\rho_{k+1}(z)\cdot z)=F(\rho_{k+1}(z)\rho_k^{-1}(z)\cdot \rho_k(z) \cdot z).\]
It can be seen that the action of the Maurer--Cartan element $K_k=\rho_{k+1}\rho_k^{-1}$ provides a mechanism for any invariant, written in terms of the components of the invariant $\rho_k(z)\cdot z$, to be expressed in terms of the
components of the invariant $\rho_{k+1}(z)\cdot z$. Such a mechanism is an 
example of a \textit{syzygy}, which we define next.}}

\begin{definition}[Syzygy]
A syzygy on a set of invariants is a relation between invariants that expresses functional dependency.
\end{definition}

Hence a syzygy on a set of invariants is a function of invariants,
which is identically zero when the invariants are expressed in terms of the underlying variables (in this case, $z\in \man $).

The key idea is to use the Maurer--Cartan group elements, which are well-adapted to studying difference equations, to express \textit{all} invariants in terms of a small generating set. Using \eqref{invariant} and \eqref{Kk} we have
\begin{equation}\label{Invrec1d}
  K_k\cdot I_{k,j} = {{ \rho_{k+1}\rho_{k}^{-1}\cdot \rho_k \cdot z_j = \rho_{k+1} \cdot z_j =}}\ I_{k+1,j}, 
\end{equation}
and iterating this, we have $K_{k+1}K_k\cdot I_{k,j} = I_{k+2,j}$, and so on, which leads to the following result.

\begin{theorem}[\citet{MBW}, Proposition 3.11]\label{MCrecThm}
Given a right discrete moving frame $\rho$, the components of $K_k$, together with the set of all \textit{diagonal invariants}, $I_{j,j}=\rho_j(z)\cdot z_j$, generate all other invariants.
\end{theorem}

{{We refer to the difference identities, or syzygies, (\ref{Invrec1d}) as \textit{recurrence relations} for the invariants. It is helpful to extend slightly the notion of a generating set from Definition \ref{DefnGenSet}.

		\begin{definition} A set of invariants is a generating set for an algebra of  difference invariants if any difference invariant in the algebra can be written as a function of elements of the generating set and their shifts.
		\end{definition}
}}

For a right difference moving frame, the identities $I_{j,j}=\es_jI_{0,0}$ and $K_k=\es_kK_0$ hold, so Theorem \ref{MCrecThm} reduces to the following result.

\begin{theorem}\label{diffMCrecThm}
	Given a right difference moving frame $\rho$, the set of all invariants is generated by the set of components of $K_0=\rho_1\rho_0^{-1}$ and $I_{0,0}=\rho_0(z)\cdot z_0$.
\end{theorem}

Note that as $K_0$ is invariant, the Replacement Rule gives the following useful identity:
\begin{equation}\label{invk0}
K_0=\iup_0(\rho_1),
\end{equation}
where $\iup_0$ denotes invariantization with respect to the frame $\rho_0$.

\begin{example}(Example \ref{simpleExOne} cont.)\  
As the Lagrangian $\upL$ in \eqref{babylagrangian} is second-order, the Euler--Lagrange equations define a subspace of the prolongation space $\man=P_n^{(-2,2)}(\mathbb{R}^2)$, which we use for the remainder of this example. Early in this section, we found a continuous moving frame $\rho$ for $P_n^{(0,1)}(\mathbb{R}^2)$;
this can be used to construct a difference moving frame on $\man$, setting
\begin{equation}
\label{difframe1}
\rho_0
=\left(\begin{array}{ccc}\ds\frac1{(u_1-u_0)^3} &0&-\,\ds\frac{x_0}{(u_1-u_0)^3}\\ 0&\ds\frac1{u_1-u_0} &  -\,\ds\frac{u_0}{u_1-u_0}\\ 0&0&1\end{array}\right)
\end{equation}
and $\rho_k=\es_k\rho_0$.
It is helpful to review the recurrence relations obtained earlier in the light of Equation (\ref{Invrec1d}). By definition,
$$ \left(\begin{array}{c} I_{0,j}^x\\ I_{0,j}^u\\ 1\end{array}\right)=\iup_0\! \left(\begin{array}{c} x_j\\ u_{j}\\ 1\end{array}\right)=\rho_0\! \left(\begin{array}{c} x_j\\ u_{j}\\ 1\end{array}\right)$$
and therefore
\begin{equation}\label{examplebabyframeA:B}
\left(\begin{array}{c} \es I_{0,j}^x\\ \es I_{0,j}^u\\ 1\end{array}\right)=\rho_1\! \left(\begin{array}{c} x_{j+1}\\ u_{j+1}\\ 1\end{array}\right)
=\left(\rho_1 \rho_0^{-1}\right) \rho_0\! \left(\begin{array}{c} x_{j+1}\\ u_{j+1}\\ 1\end{array}\right)
=K_0 \left(\begin{array}{c} I_{0,j+1}^x\\ I_{0,j+1}^u\\1\end{array}\right).
\end{equation}
Calculating the matrix $K_0=\rho_1\rho_0^{-1}=\iup_0(\rho_1)$ yields
\begin{equation}\label{examplebabyframeA:C} K_0 = \left(\begin{array}{ccc} \ds\frac1{(\kappa-1)^3}& 0&-\,\ds\frac{\eta}{(\kappa-1)^3}  \\ 0& \ds\frac1{\kappa-1} &-\,\ds\frac1{\kappa-1} \\  0&0&1\end{array}\right),\end{equation}
where $\eta$ and $\kappa$ are defined in \eqref{kappaandeta}.
Clearly, equations (\ref{examplebabyframeA:A}) and (\ref{examplebabyframeA:B}) are consistent.

The Maurer--Cartan invariants for this example are the components of $K_0$ and their shifts.
By Theorem \ref{diffMCrecThm}, the algebra of invariants is generated by $\eta$, $\kappa$ and their shifts, because both components of $I_{0,0}=\rho_0\cdot (x_0,u_0)$ are zero.
\end{example}

A complete discussion of Maurer--Cartan invariants for discrete moving frames, with their recurrence relations and discrete syzygies, is given in \citet{GMBMAN}.
\subsection{Differential--difference invariants and the differential--difference syzygy}

We aim to obtain the Euler--Lagrange equations in terms of the invariants, and also the form of the conservation laws in terms of invariants and a frame.
A key ingredient of our method will be 
a {\em differential--difference syzygy} between differential and difference invariants, 
which will feature prominently in our formulae.

Given any smooth path $t\mapsto z(t)$ in the space $\man =M^N$, consider the induced group action on the path and its tangent. We extend the group action to the dummy variable $t$ trivially, so that $t$ is invariant.
The action is extended to the first-order jet space of $\man$ as follows:
$$g\cdot\frac{{\rm d} z(t)}{{\rm d}t} = \frac{{\rm d} \left(g\cdot z(t)\right)}{{\rm d}t}\,.$$
If the action is free and regular on $\man $, it will remain so on the jet space and we may use the same frame to find the first-order differential invariants
\begin{equation}\label{FOdiffinvariant}
I_{k,j;\,t}(t):=\rho_k(z(t))\cdot \frac{{\rm d}  z_j(t)}{{\rm d}t}\,.
\end{equation}
Let $I_{k,j}(t)$ denote the restriction of $I_{k,j}$ to the path $z(t)$. The frame depends on $z(t)$, so, in general,
$$I_{k,j;\,t}(t)\neq \frac{{\rm d}}{{\rm d}t} I_{k,j}(t).$$

The particular {{differential--difference}} syzygy that we will need to calculate the invariantized variation of the Euler--Lagrange equations
concerns the relationship between the $t$-derivative of the discrete invariants $I_{k,j}(t)$
and the differential--difference invariants $I_{k,j;\,t}(t)$. It takes the form
\begin{equation}\label{mainDDSyz}
\frac{{\rm d}}{{\rm d}t} \bkappa = \mathcal{H}\bom,
\end{equation}
where $\bkappa$ is a vector of generating  invariants, $\mathcal{H}$ is a linear difference operator with coefficients 
that are functions of $\kappa$ and its shifts, and $\bom$ is a vector of generating first order differential invariants
of the form \eqref{FOdiffinvariant}.

If the generating discrete invariants are known, the syzygies can be found by direct differentiation followed by the Replacement Rule (Theorem \ref{reprule}). Recurrence relations for the differential invariants are obtained 
in a manner analogous to those of the discrete invariants, as illustrated in the running example below. This allows one to write the syzygy in terms of a set of generating differential invariants.
Another method is to differentiate the Maurer--Cartan matrix as follows.
Given a matrix representation for the right frame $\rho_k$, restricted throughout the following to the path $z(t)$, apply the product rule to the definition of $K_k$ to obtain
\begin{equation}\label{DiffKeqn}
\frac{{\rm d}}{{\rm d}t} K_k=\frac{{\rm d}}{{\rm d}t}\,\left( \rho_{k+1}\rho_{k}^{-1}\right) = \left(\frac{{\rm d}}{{\rm d}t} \rho_{k+1}\right) \rho_{k+1}^{-1}  K_k -  K_k \left(\frac{{\rm d}}{{\rm d}t} \rho_{k} \right)\rho_{k}^{-1}.
\end{equation} 
This motivates the following definition.
\begin{definition}[Curvature Matrix]\label{Nndefn} The curvature matrix $N_k$ is given by
 \begin{equation}\label{curvMatDef}
N_k=\left(\frac{{\rm d}}{{\rm d}t}\, \rho_k\right) \rho_k^{-1}
\end{equation}
when $\rho_k$ is in matrix form.
\end{definition}
It can be seen that for a right frame,  $N_k$ is an invariant matrix that involves the first order differential invariants.
The above derivation applies to all discrete moving frames. For a difference frame, moreover, $N_k=\es_kN_0$ and (\ref{DiffKeqn}) simplifies to the set of shifts of a generating syzygy,
\begin{equation}\label{DiffKeqnN}\frac{{\rm d}}{{\rm d}t} K_0=(\es N_0) K_0-K_0N_0.\end{equation}
As $N_0$ is invariant, the Replacement Rule yields another useful identity:
\begin{equation}\label{invn0}
N_0=\iup_0\!\left(\frac{{\rm d}}{{\rm d}t}\, \rho_0\right).
\end{equation}

Equating components in (\ref{DiffKeqnN})  yields syzygies relating the derivatives of the Maurer--Cartan invariants to the first order differential invariants and their shifts.

Finally, we need the differential--difference syzygies for the remaining generating invariants, the diagonal invariants
(see Theorem \ref{diffMCrecThm}). For a linear (matrix) action,
\begin{equation}\label{diffdiffdiagsyz}
\frac{{\rm d}}{{\rm d}t} I_{0,0}(t)
= \left(\frac{{\rm d}}{{\rm d}t} \rho_0\right) \rho_0^{-1}\cdot  \left(\rho_0 \cdot z_0(t)\right) +\rho_0\cdot \frac{{\rm d}}{{\rm d}t} z_0(t)
=N_0 I_{0,0}(t) + I_{0,0;\,t}(t).
\end{equation}
For nonlinear actions, the techniques described in the text \citet{Mansfield:2010aa} may be modified to accommodate difference moving frames.

In all examples in this paper, the diagonal invariants $I^\alpha_{0,0}$ are normalized to be constants; nevertheless, there are examples where this need not hold. In some circumstances, it is necessary to chose a normalization that makes off-diagonal invariants constants, in which case some diagonal invariants may depend on $z(t)$.

\begin{example}(Example \ref{simpleExOne} cont.)\  
We now turn our attention to the differential invariants for our running example. Writing $x_j=x_j(t)$ and $u_j=u_j(t)$, etc., the action on the derivatives $x_j'={\rm d} x_j/{\rm d}t$, $u_j'={\rm d} u_j/{\rm d}t$ is induced by the chain rule, as follows:
$$g\cdot x_j' = \frac{{\rm d} \, \left(g\cdot x_j \right)}{{\rm d}\, \left(g\cdot t\right)}= \frac{{\rm d}\,\left(g\cdot x_j \right)}{{\rm d} t} = \lambda^3  x_j', $$
and similarly,  $$g\cdot u_j' = \lambda u_j'.$$
Define  
 \begin{equation}\label{difinvariantsexample}
I^x_{0,j;\,t}=\rho_0\cdot x_j'=\frac{x_j'}{\left(u_1-u_0\right)^3},\qquad I^u_{0,j;\,t}=\rho_0\cdot u_j'=\frac{u_j'}{u_1-u_0}.
\end{equation}
We first obtain  recurrence relations for the $I^x_{k,j;\,t}$ and $I^u_{k,j;\,t}$. As 
$$\es I^x_{0,j;\,t}=\es(\rho_0\cdot x_j')\ {{= \rho_1\cdot x_{j+1}'}}=\left(\rho_1\rho_0^{-1}\right)\rho_0\cdot x_{j+1}'=K_0\cdot I^x_{0,j+1;\,t},$$ and similarly for  $I^u_{0,j;\,t}$, it follows that
\begin{equation}\label{examplebabyframeA:D}
\es I^x_{0,j;\,t} = \frac{I^x_{0,j+1;\,t}}{\left(\kappa-1\right)^3},\qquad \es I^u_{0,j;\,t} = \frac{I^u_{0,j+1;\,t}}{\kappa-1}.
\end{equation}
In the same way, one can use the shift operator and $\rho_k\rho_0^{-1}=K_{k-1}K_{k-2}\cdots K_0$ to obtain all $I^x_{k,j;\,t}$, $I^u_{k,j;\,t}$ in terms of the generating Maurer--Cartan invariants, 
$$\sigma^x:=I^x_{0,0;\,t}=\iup_0(x_0')\ {{=\rho_0\cdot x_0'=\frac{x_0'}{\left(u_1-u_0\right)^3}}},$$
$$\sigma^u:=I^u_{0,0;\,t}=\iup_0(u_0')\ {{=\rho_0\cdot u_0'=\frac{u_0'}{u_1-u_0}}},$$ and their shifts.
We now obtain the differential--difference syzygies (\ref{DiffKeqnN}).
The simplest way to calculate $N_0$ is to use the identity \eqref{invn0}:
\begin{equation}\label{examplebabyframeA:E}
N_0=\iup_0\left(\displaystyle\frac{{\rm d}}{{\rm d}t} \rho_0\right) =
\left(\begin{array}{ccc} -3(\iup_0(u_{1}')-\iup_0(u_{ 0}'))&0&-\iup_0(x_{0}')\\ 0&-(\iup_0(u_{1}')-\iup_0(u_{ 0}')) & - \iup_0(u_{ 0}')\\ 0&0&0 \end{array}
\right).
\end{equation}
Other methods are detailed in \citet{Mansfield:2010aa}.
Now use the recurrence relations (\ref{examplebabyframeA:D}) to obtain the differential invariants in (\ref{examplebabyframeA:E})
in terms of $\sigma^x$, $\sigma^u$ and their shifts:
\begin{equation}\label{N0fin}
 N_0=\left(\begin{array}{ccc}-3\left((\kappa-1)\es\:\! \sigma^u-\sigma^u\right) & 0 & -\sigma^x\\
            0& -\left((\kappa-1)\es\:\! \sigma^u-\sigma^u\right) & -\sigma^u\\
            0&0&0
           \end{array}\right).
\end{equation}
Inserting (\ref{examplebabyframeA:C}) and (\ref{N0fin}) into
(\ref{DiffKeqnN}) yields, after equating components and simplifying,
\begin{equation}\label{examplebabyframeA:G}
\begin{array}{rcl}
\displaystyle\frac{{\rm d} \eta}{{\rm d}t} &=& \left[(\kappa-1)^3\, \es -\id\;\!\right]\sigma^x + 3\eta\left[\;\!\id-(\kappa-1)\;\!\es\,\right]\sigma^u,\\[12pt]
\displaystyle\frac{{\rm d} \kappa}{{\rm d}t} &=& \left( \kappa-1\right)\left[\;\!\id-\kappa\;\! \es +( \kappa_1 -1)\,\es_2\;\!\right]\sigma^u.
\end{array}
\end{equation}
Therefore, the differential--difference syzygy between the generating difference invariants, $\eta$ and $\kappa$, and the generating differential invariants, $\sigma^x$ and $\sigma^u$, 
can be put into the canonical form
$$ \displaystyle\frac{{\rm d}}{{\rm d}t} \left(\begin{array}{c} \eta \\ \kappa\end{array}\right) = \mathcal{H}\left(\begin{array}{c}\sigma^x \\ \sigma^u\end{array}\right), $$
where $\mathcal{H}$ is a linear difference operator whose coefficients depend only on the generating difference invariants and their shifts.
\end{example}

\section{The  Euler--Lagrange equations for a Lie group invariant Lagrangian}\label{inveul}
We are now ready to present our first main result, the calculation of the Euler--Lagrange equations, in terms of invariants, for a Lie group invariant difference Lagrangian.
We emulate the calculation of the Euler--Lagrange equations given in \S\ref{introdiffcalc}, but use the computational techniques developed above for difference invariants.

First, recall the summation by parts formula (\ref{sumpartdef}). This leads to the following similar definition.
\begin{definition}\label{adopdef} Given a linear difference operator $\mathcal{H}=c_j \es_j$,  the adjoint operator $\mathcal{H}^*$ is defined by
$$\mathcal{H}^*(F)=\es_{-j}(c_j F)$$
and the associated boundary term $A_{\mathcal{H}}$ is defined by $$  F\mathcal{H}(G)-\mathcal{H}^*(F)G = (\es-\id)(A_{\mathcal{H}}(F,G)),$$
for all appropriate expressions $F$ and $G$.
\end{definition}

Now suppose we are given a group action $G\times {M}\rightarrow {M}$ and that we have found a difference frame for this action. 
Any group-invariant Lagrangian $\upL(n,\mathbf{u}_0,\dots,\mathbf{u}_J)$ can be written, in terms of the generating invariants $\bkappa$ and their shifts $\bkappa_j=\es_j\bkappa$, as $L(n,\bkappa_0,\dots, \bkappa_{J_1})$ for some $J_1$; we adopt this notation from here on. For consistency, we drop the argument from the associated functional, setting
\[
\mathcal{L}=\sum \upL(n,\mbu_0,\dots, \mbu_{J})=\sum L(n,\bkappa_0,\dots, \bkappa_{J_1}).
\]

\begin{theorem}[Invariant Euler--Lagrange Equations]\label{mainThmELeqns}
Let $\mathcal{L}$ be a Lagrangian functional whose invariant Lagrangian is given in terms of the generating invariants as
$$\mathcal{L}=\sum L(n,\bkappa_0,\dots, \bkappa_{J_1}),$$
and suppose that the differential--difference syzygies are
$$ \frac{{\rm d} \bkappa}{{\rm d}t} = \mathcal{H}\bom.$$
Then (with $\cdot$ denoting the sum over all components)
\begin{equation}\label{ELinvresult}
\eubu(\upL)\cdot \mbu_0'=\big(\mathcal{H}^*\eubka(L)\big)\cdot\bom,
\end{equation}
where $\eubka(L)$ is the difference Euler operator with respect to $\bkappa$.
Consequently, the invariantization of the original Euler--Lagrange equations is
\begin{equation}\label{invEL}
\iup_0\big(\eubu(\upL)\big)=\mathcal{H}^*\eubka(L).
\end{equation}
\end{theorem}
\begin{proof} In order to effect the variation, set $\mbu=\mbu(t)$
and compare
\begin{equation}
\frac{{\rm d}}{{\rm d}t}\, \mathcal{L}  = \sum \big\{\eubu(\upL)\cdot \mbu_0'+(\es-\id)(A_\mbu)\big\}
\end{equation}
with the same calculation in terms of the invariants. This gives $\upd \mathcal{L}/\upd t=\sum \upd L/\upd t$, where
\begin{equation}\label{ELcalc1} \begin{array}{rcl}
\displaystyle\frac{{\rm d} L}{{\rm d}t}  &=& \ds\frac{\partial L}{\partial \kappa^\alpha_j}\, \ds\frac{{\rm d} \kappa^\alpha_j}{{\rm d}t}\\[12pt]
&=& \ds\frac{\partial L}{\partial \kappa^\alpha_j}\, \es_j \ds\frac{{\rm d}  \kappa^\alpha}{{\rm d}t}\\[10pt]
&=& \left( \es_{-j} \ds\frac{\partial L}{\partial \kappa^\alpha_j}\right) \ds\frac{{\rm d}  \kappa^\alpha}{{\rm d}t} + (\es -\id)(A_{\bkappa} )\\[10pt]
&=& \eubka(L)\cdot \ds\frac{{\rm d}  \bkappa}{{\rm d}t} +(\es -\id)(A_{\bkappa} )\\[10pt]
&=& \eubka(L)\cdot \mathcal{H}\bom +(\es -\id)(A_{\bkappa} )\\[10pt]
&=& \big(\mathcal{H}^*\eubka(L)\big)\cdot\bom + (\es -\id)\{A_{\bkappa} +A_{\mathcal{H}}\}.
\end{array}
\end{equation}
The divergence terms arising from the first and second summations by parts are $(\es -\id)A_{\bkappa} $ and $(\es -\id)A_{\mathcal{H}}$ respectively. (Note that $A_{\bkappa} $ is linear in the $\upd\kappa^\alpha/\upd t$ and their shifts,
while $A_{\mathcal{H}}$ is linear in the $\sigma^{\alpha}$ and their shifts.)
By the Fundamental Lemma of the Calculus of Variations, the identity (\ref{ELinvresult}) holds. To derive \eqref{invEL}, apply $\iup_0$ to \eqref{ELinvresult} and compare components of $\bom$.
\end{proof}

Consequently, the original Euler--Lagrange equations, in invariant form, are equivalent to
\[
\mathcal{H}^*\eubka(L)=0.
\]

\begin{example}(Example \ref{simpleExOne} cont.)\  
The invariant Lagrangian in our example is of the form
$$\mathcal{L}=\sum L(\eta, \kappa, \es\kappa).$$
Using $\eta_j=\es_j\eta$ and $\kappa_j=\es_j\kappa$ henceforth, we write Equation (\ref{examplebabyframeA:G}) as
\begin{equation}\label{exposEGtderivsinvs}\begin{array}{rcl}
\displaystyle\frac{{\rm d} \eta}{{\rm d}t} &=& \mathcal{H}_{11}\;\!\sigma^x + \mathcal{H}_{12}\;\!\sigma^u,\\[10pt]
\displaystyle\frac{{\rm d} \kappa}{{\rm d}t} &=& \mathcal{H}_{22}\;\! \sigma^u,\end{array}\end{equation}
with 
$$\begin{array}{rcl}
\mathcal{H}_{11}&=&(\kappa-1)^3\, \es -\id,\\
\mathcal{H}_{12}&=&3\eta\{\id-(\kappa-1)\,\es\},\\
\mathcal{H}_{22}&=& \left( \kappa-1\right)\{\id-\kappa \es +(\kappa_1 -1)\,\es_2\}.\end{array}$$
By Theorem \ref{mainThmELeqns}, the invariantized Euler--Lagrange equations are
$$ \mathcal{H}_{11}^*\euet(L)=0,\qquad \mathcal{H}_{12}^*  \euet(L)  + \mathcal{H}_{22}^* \euka(L)=0,$$
where
\[
\begin{aligned}
&  \mathcal{H}_{11}^*=(\kappa_{-1}-1)^{3}\,\es_{-1}-\id,\\
&\mathcal{H}_{12}^*=3\eta\,\id-3\eta_{-1}(\kappa_{-1}-1)\,\es_{-1}, \\
& \mathcal{H}_{22} ^*=(\kappa-1)\id-\kappa_{-1}(\kappa_{-1}-1) \es_{-1} + (\kappa_{-2}-1)(\kappa_{-1}-1)\es_{-2}.
\end{aligned}
\]
The particular Lagrangian \eqref{babylagrangian} amounts to $L=\eta(\kappa-1)^{-3/2}$, so
\[
\euet=(\kappa-1)^{-3/2},\qquad\euka=-\tfrac{3}{2}\,\eta(\kappa-1)^{-5/2}.
\]
Consequently, the invariantized Euler--Lagrange equations are
\begin{equation}\label{invelex1}
(\kappa_{-1}-1)^{3/2}-(\kappa-1)^{-3/2}=0,
\end{equation}
\begin{equation}\label{invelex2}
\tfrac{3}{2}\big\{\eta(\kappa-1)^{-3/2}-\eta_{-1}(\kappa_{-1}-1)^{-1/2}+\eta_{-1}(\kappa_{-1}-1)^{-3/2}-\eta_{-2}(\kappa_{-1}-1)(\kappa_{-2}-1)^{-3/2}\big\}=0.
\end{equation}
Assuming that $L$ is real-valued ($\kappa>1$), the general solution of \eqref{invelex1} is
\begin{equation}\label{kapsol}
\kappa=1+k_1^{2(-1)^n}=1+\tfrac{1}{4}\left[k_1+k_1^{-1}+(k_1-k_1^{-1})(-1)^n\right]^2,
\end{equation}
where $k_1$ is an arbitrary nonzero constant. Therefore \eqref{invelex2} simplifies to
\[
k_1^{3(-1)^{n+1}}\eta+\left(k_1^{3(-1)^n}-k_1^{(-1)^n}\right)\eta_{-1}-k_1^{5(-1)^{n+1}}\eta_{-2}=0,
\]
whose general solution is
\begin{align}\label{etasol}
\eta&=k_1^{3(-1)^n}\left\{k_2\left((n+1)k_1^{(-1)^{n+1}}-nk_1^{(-1)^n}\right)+k_3(-1)^n\right\},
\end{align}
where $k_2$ and $k_3$ are arbitrary constants.
\end{example}

\section{On infinitesimals and the adjoint action}\label{moregrpsec}
To state our results concerning the conservation laws, it is necessary to use the infinitesimal generators of a Lie group action on a manifold, together with the adjoint representation of the Lie group.

\begin{definition}\label{usualinfdefn}
Let $G\times { U} \rightarrow { U}$   be a smooth local Lie group action. If $\gamma(t)$ is a path in $G$ with $\gamma(0)=e$, the identity element in $G$, then
\begin{equation}\label{usualinfdefnEqn}
\mbv=\frac{{\rm d}}{{\rm d}t}\,\Big\vert_{t=0} \gamma(t)\cdot \mbu
\end{equation}
is called the infinitesimal generator of the group action at $\mbu\in U$, in the direction $\gamma\!\phantom{.}'(0)\in T_eG$,
where $T_eG$ is the tangent space to $G$ at $e$. In coordinates, the components of the infinitesimal generator are $\phi^\alpha=\mbv(u^\alpha)$, so
\[
\mbv=\phi^\alpha\frac{\partial}{\partial u^\alpha}\,.
\]
\end{definition}

The infinitesimal generator is extended to the prolongation space $\man =P^{(J_0,J)}_n(U)$ by the prolongation formula
\[
\mbv(u^\alpha_j)=\left.\frac{\textrm{d}}{\textrm{d}t}\right\vert_{t=0} \gamma(t)\cdot u^\alpha_{j}=\phi^\alpha_j=\es_j \phi^\alpha_0,\qquad J_0\leq j\leq J,
\]
\citep[see][]{Hydon}. In coordinates, the prolonged infinitesimal generator is
\[
\mbv=\phi^\alpha_{j}\frac{\partial}{\partial u^\alpha_{j}}\,.
\]

\begin{lemma} If a Lagrangian $\upL[\mbu]$ is invariant under the group action  $G\times \man \rightarrow \man $, the components of the infinitesimal generator of the group action
given by Definition \ref{usualinfdefn}
form the characteristic of a variational symmetry of $\upL[\mbu]$, as defined in Definition \ref{infpart1}.
\end{lemma}

\begin{proof}  The Lagrangian $\upL$ is invariant, so $$\upL(\mbu_0, \mbu_1, \ldots, \mbu_J)=\upL(g\cdot \mbu_0, g\cdot \mbu_1, \ldots, g\cdot \mbu_J)$$ for all $g$. 
Thus
$$
0=\left.\frac{\textrm{d}}{\textrm{d}t}\right\vert_{t=0} \upL\left(\gamma(t)\cdot \mbu_0, \gamma(t)\cdot \mbu_1, \dots \right)=\mbv({\upL}) = \phi^\alpha_{j}\frac{\partial \upL}{\partial u^\alpha_{j}}.
$$
By Definition \ref{infpart1}, the components $\phi^\alpha$ of the infinitesimal generator are the components of the characteristic of a variational symmetry of $\upL$.
\end{proof}

Each infinitesimal generator is determined by $\gamma\!\phantom{.}'(0)\in T_eG$; the remainder of the path in $G$ is immaterial. However, $T_eG$ is isomorphic to the Lie algebra $\mathfrak{g}$, which is the set of right-invariant vector fields on $G$. Right-invariance yields a Lie algebra homomorphism from $\mathfrak{g}$ to the set $\mathcal{X}$ of infinitesimal generators of symmetries (see \citet{Obook} for details). If the group action is faithful, this is an isomorphism.

The $R$-dimensional Lie group $G$ can be parametrized by $\mba=(a^1, \dots , a^R)$ in a neighbourhood of the identity, $e$, so that the general group element is $\Gamma(\mba)$, where $\Gamma(\mathbf{0})=e$. Given local coordinates $\mbu=(u^1,\dots,u^q)$ on $U$, let $\widehat{\mbu}=\Gamma(\mba)\cdot\mbu$. By varying each independent parameter $a^r$ in turn, the process above yields $R$ infinitesimal generators,
\begin{equation}\label{infgen}
\mbv_r=\xi^\alpha_r(\mbu)\partial_{u^\alpha},\quad\text{where}\quad\xi^\alpha_r=\frac{\partial \hat{u}^\alpha}{\partial a^r}\Big|_{\mba=\mathbf{0}}.
\end{equation}
These form a basis for $\mathcal{X}$.

As $\mathcal{X}$ is homomorphic to $\mathfrak{g}$, the adjoint representation of $G$ on $\mathfrak{g}$ gives rise to the adjoint representation of $G$ on $\mathcal{X}$. Given $g\in G$, the adjoint representation $Ad_g$ is the tangent map on $\mathfrak{g}$ induced by the conjugation $h\mapsto gh g^{-1}$. The corresponding adjoint representation on $\mathcal{X}$  is expressed by a matrix, $\myAd(g)$, which is most conveniently obtained as follows\footnote{{See the Appendix for an alternative construction using Lie algebra structure constants.}}. Having calculated a basis for $\mathcal{X}$,
\[
\mbv_r=\xi^\alpha_r(\mbu)\,\partial_{u^\alpha},\qquad r=1,\dots,R,
\]
let $\widetilde\mbu=g\cdot\mbu$ and define
\[
\widetilde\mbv_r=\xi^\alpha_r(\widetilde\mbu)\,\partial_{\tilde u^\alpha},\qquad r=1,\dots,R.
\]
Now express each $\mbv_r$ in terms of $\widetilde\mbv_1,\dots,\widetilde\mbv_R$ and determine $\myAd(g)$ from the identity
\begin{equation}\label{addef}
(\mbv_1\ \cdots\ \mbv_R)=(\widetilde\mbv_1\ \cdots\ \widetilde\mbv_R)\!\:\myAd(g).
\end{equation}

\begin{example}(Example \ref{simpleExOne} cont.)\  
	For our running example, the group parameters are $\lambda,$ $a$ and $b$ with identity $(\lambda,a,b)=(1,0,0)$. Choosing $a^1=\ln(\lambda), a^2=a$ and $a^3=b$, so that the identity corresponds to $\mba=\mathbf{0}$, one obtains the following basis for $\mathcal{X}$:
\[
\mbv_1=3x\partial_{x} + u\partial_{u},\qquad\mbv_2=\partial_x,\qquad\mbv_3=\partial_u.
\]	
Recall that the action of a fixed group element $g$, parametrized by $(\lambda,a,b)$, gives 
\[
(\tilde x,\tilde u)=(\lambda^3 x +a , \lambda u +b).
\]
Therefore, by the standard change-of-variables formula,
\[
\mbv_1=3(\tilde x-a)\partial_{\tilde x}+(\tilde u-b)\partial_{\tilde u},\qquad \mbv_2=\lambda^3\partial_{\tilde x},\qquad \mbv_3=\lambda\partial_{\tilde u}.
\]
Consequently,
\[
(\mbv_1\ \mbv_2\ \mbv_3)=(\widetilde\mbv_1\ \widetilde\mbv_2\ \widetilde\mbv_3)\!\:\myAd(g),\quad\text{where}\quad \myAd(g)=
\left(\begin{array}{ccc} 1 & 0&0\\ -3a&\lambda^3&0\\-b&0&\lambda\end{array}\right).
\]
\end{example}

Regarding the infinitesimal generators as differential operators and applying the identity \eqref{addef} to each $\tilde u^\alpha$ in turn, one obtains
\begin{equation}\label{addefut}
(\mbv_1(\tilde u^\alpha)\ \cdots\ \mbv_R(\tilde u^\alpha))=(\xi^\alpha_1(\widetilde\mbu)\ \cdots\ \xi^\alpha_R(\widetilde\mbu))\!\:\myAd(g).
\end{equation}
This yields a useful matrix identity. Define the matrix of characteristics to be the $q\times R$ matrix
\begin{equation}\label{matrixInfsdefn}
\Phi(\mbu)=\big(\xi^\alpha_r(\mbu)\big).
\end{equation}
Then, by the chain rule, \eqref{addefut} amounts to
\begin{equation}\label{addefmat}
\left(\frac{\partial \widetilde{\mbu}}{\partial
	\mbu}\right)\Phi(\mbu)=\Phi(\widetilde\mbu)\!\:\myAd(g),
\end{equation}
where $\left(\partial \widetilde{\mbu}/\partial
	\mbu\right)$ is the Jacobian matrix.
This identity is extended to prolongation spaces with coordinates $z=(\mbu_{J_{0}},\dots,\mbu_J)$, where $J_{0}\leq 0$ and $J\geq 0$, as follows. Define the matrix of prolonged infinitesimals to be
\[
\Phi(z)=\left(
\begin{array}{c}
\Phi(\mbu_{J_0})\\
\vdots\\
\Phi(\mbu_{J})
\end{array}\right).
\]
The infinitesimal generators $\mbv_r$, prolonged to all variables in $z$, satisfy \eqref{addef}, where the tilde now denotes replacement of $z$ by $g\cdot z$. Applying this identity to $g\cdot z$ gives
\begin{equation}\label{infid}
\left(\frac{\partial (g\cdot z)}{\partial
	z}\right)\Phi(z)=\Phi(g\cdot z)\!\:\myAd(g).
\end{equation}

\section{Conservation laws}\label{Conlaws}

In general, the conservation laws are not invariant. However, as we will show,  they are equivariant; indeed, they can be written in terms of invariants and the frame.
Our key result is that for difference frames, the $R$ conservation laws can be written in the form
$$(\es-\id)\{ V(I)\!\:\myAd(\rho_0)\}=0$$
where $\myAd(\rho_0)$ is the adjoint representation of $\rho_0$ and $ V(I)=(V_1\ \cdots\ V_R)$ is a row vector of invariants.

In the standard (that is, not invariantized) calculation of the Euler--Lagrange equations and boundary terms, suppose that the dummy variable $t$ effecting the variation 
is a group parameter for $G$, under which the Lagrangian is invariant. Then
the resulting boundary terms yield conservation laws; this is the difference version of Noether's theorem. So it is useful to identify $t$ with a group parameter by considering the following path in $G$:
\begin{equation}\label{grpath}  
t\mapsto \gamma_r(t)=\Gamma\left(a^1(t),\dots, a^R(t)\right),\quad\text{where}\quad a^r(t)=t\quad\text{and}\quad a^l(t)=0,\ l\ne r;
\end{equation}
{{recall from \S\ref{moregrpsec} that $\mathbf{a}\mapsto \Gamma(\mathbf{a})$ expresses the general group element in terms of the coordinates $\mathbf{a}$.}}
On this path, each $(\mbu_0)'$  at $t=0$ is an infinitesimal generator, from \eqref{infgen}.

For the invariantized calculation, we follow essentially the same route to our result, identifying the dummy variable effecting the variation with each group parameter in turn. The proof of Theorem \ref{mainThmELeqns} uses the identity
\begin{equation}\label{invdt}
\frac{{\rm d}}{{\rm d}t}\,L\left(n,\bkappa,\dots, \es_{J_1}(\bkappa)\right) 
=\big(\mathcal{H}^*\eubka(L)\big)\cdot\bom + (\es -\id)\{A_{\bkappa} +A_{\mathcal{H}}\}.
\end{equation}
Recall that $A_{\bkappa} $ is linear in $\upd\kappa^\alpha/\upd t$ and their shifts, while $A_{\mathcal{H}}$ is linear in the $\sigma^{\alpha}$ and their shifts. As $t$ is a group parameter and each $\kappa^\alpha$ is invariant, 
$\upd \kappa^\alpha/\upd t =0$. Thus, \eqref{invdt} reduces to
\begin{equation}\label{invnoe1}
\big(\mathcal{H}^*\eubka(L)\big)\cdot\bom + (\es -\id)A_{\mathcal{H}}=0,
\end{equation}
so $(\es -\id)A_{\mathcal{H}}=0$ on all solutions of the invariantized Euler--Lagrange equations $\mathcal{H}^*\eubka(L)=0$. We now derive the conservation laws from this condition.

\begin{theorem}\label{thmi} Suppose that the conditions of Theorem  \ref{mainThmELeqns} hold. Write
	\[
	 {A}_{\mathcal{H}}=\mathcal{C}^{j}_{\alpha}\es_j(\sigma^{\alpha}),
	\]
where each $\mathcal{C}^{j}_{\alpha}$ depends only on $n, \bkappa$ and its shifts.
Let $\Phi^{\alpha}(\mbu_0)$ be the row of the matrix of characteristics corresponding to the dependent variable $u^{\alpha}_0$ and denote its invariantization by $\Phi^{\alpha}_0(I)=\Phi^{\alpha}(\rho_0\cdot\mbu_0)$. 
Then the $R$ conservation laws in row vector form amount to
\begin{equation}\label{conlawsfinala}
\mathcal{C}_{\alpha}^{j}\es_j\{\Phi^{\alpha}_0(I)\!\:\myAd\left(\rho_{0}\right)\}=0.
\end{equation}
That is, to obtain the conservation laws, it is sufficient to make the replacement
\begin{equation}\label{InvAdrepEqn}
\sigma^\alpha \mapsto \{\Phi^{\alpha}(g\cdot \mbu_0)\!\:\myAd(g)\} \big\vert_{g=\rho_0}.
\end{equation}
in $A_{\mathcal{H}}$.
\end{theorem}

\begin{proof}
	Recall that
	\begin{equation}\sigma^\alpha=\rho_0\cdot (u^\alpha_0)'=\left(\frac{{\rm d}}{{\rm d}t}\,g\cdot {u^\alpha_{0}}\right)\Big\vert_{g=\rho_0}.
	\end{equation}
	To obtain the conservation laws, conflate $t$ with the group parameter $a^r$, making the replacement
	\begin{equation}\rho_0\cdot (u_0^\alpha)' \mapsto \frac{{\rm d}}{{\rm d}t} \Big\vert_{t=0} \rho_0\cdot \gamma_r(t)\cdot u_0^\alpha\end{equation}
	in $A_{\mathcal{H}}$, where $\gamma_r(t)$ is the path defined in \eqref{grpath}. For any $g\in G$,
	\begin{equation}\label{CLrepl1}  \begin{array}{rcl}
	\ds\frac{{\rm d}}{{\rm d}t} \Big\vert_{t=0} \left(g\cdot \gamma_r(t)\cdot u_0^\alpha\right)&=&  \left(\ds\frac{\partial \left(g\cdot \gamma_r(t)\cdot u_0^\alpha\right)}
	{\partial \left(\gamma_r(t)\cdot u_j^\beta\right)}\right)\Bigg\vert_{t=0}\left(\ds\frac{{\rm d}}{{\rm d}t} \bigg\vert_{t=0} \gamma_r(t)\cdot u_j^\beta\right)\\[20pt]
	&=&\ds\frac{\partial \left(g\cdot u_0^\alpha\right)}{\partial  u_j^\beta}\left(\ds\frac{{\rm d}}{{\rm d}t} \bigg\vert_{t=0} \gamma_r(t)\cdot u_j^\beta\right).
	\end{array}
	\end{equation}
In matrix form, (\ref{CLrepl1}) amounts to the following (taking \eqref{infid} into account):
	\[
		\ds\frac{{\rm d}}{{\rm d}t} \Big\vert_{t=0} \left(g\cdot \gamma_r(t)\cdot u_0^\alpha\right) = \left(\ds\frac{\partial (g\cdot z)}{\partial z}\,\Phi(z)\right)_{(u_0^\alpha,r)}
	= \big(\Phi(g\cdot z)\!\:\myAd(g)\big)_{(u_0^\alpha,r)}\,,
	\]
where $(u_0^\alpha,r)$ denotes the entry in the row corresponding to $u_0^\alpha$ and the $r^{\mathrm{th}}$ column. Setting $g=\rho_0$, the required replacement is
	$$\sigma^\alpha \mapsto \big(\Phi(\rho_0\cdot z)\!\:\myAd(\rho_0)\big)_{(u_0^\alpha,r)}=\big(\Phi(\rho_0\cdot \mbu_0)\!\:\myAd(\rho_0)\big)^\alpha_r\,. $$
By using each parameter $a^r$ in turn, $\sigma^\alpha$ is replaced by a row vector,
	$$\sigma^\alpha \mapsto \Phi_0^\alpha(I)\!\:\myAd(\rho_0),$$
as required.
\end{proof}

By the prolongation formula $\es_j(\rho_0)=\rho_{j}$, the conservation laws amount to
\begin{equation}
(\es-\id)\left(\mathcal{C}^{\alpha}_j\big(\es_j\Phi^{\alpha}_{0}(I)\big)\!\:\myAd\left(\rho_{j}\right)\right)=0,
\end{equation}
As  $\myAd(\rho_j)\!\:\myAd(\rho_0)^{-1}=\myAd(\rho_j\rho_0^{-1})$ is invariant, this leads to the following corollary.

\begin{corollary}\label{goodCLs}
	The conservation laws for a difference frame may be written in the form
	\begin{equation}\label{CLsgoodform} (\es-\id)\{V(I)\!\:\myAd(\rho_0)\} =0\end{equation}
	where $ V(I)=(V_1\ \cdots\ V_R)$ is an invariant row vector. Specifically,
	\begin{equation}
	V(I)=\mathcal{C}_{\alpha}^j\big(\es_j\Phi^{\alpha}_{0}(I)\big)\!\:\myAd\left(\rho_j\rho_{0}^{-1}\right).
	\end{equation}
\end{corollary}

\begin{corollary}
On any solution of the invariantized Euler--Lagrange equations,
\begin{equation}
V(I)\!\:\myAd\left(\rho_{0}\right)=c,
\end{equation}
for some constant row vector $c=(c_1\ \cdots\ c_R)$.
\end{corollary}

As the conservation laws depend only on the terms arising from $A_{\mathcal{H}}$,  the laws can be calculated for all Lagrangians in the relevant invariance class, in terms
of the $\eubka (L)$, independently of the precise form that $L=L(n,\bkappa,\dots, \es_{J_1}\bkappa)$ takes.

\begin{example}(Example \ref{simpleExOne} cont.)\  
Redoing an earlier calculation, but keeping track of the terms in $A_{\mathcal{H}}$, 
\begin{align*}
 \displaystyle\frac{{\rm d}}{{\rm d}t} L(\eta, \kappa, \es \kappa) &=\iup_0\{\eux(L)\}\,\sigma^x + \iup_0\{\euu(L)\}\,\sigma^u+(\es-\id)A_{\bkappa}\\[10pt]
&+ (\es -\id) \left( \es_{-1} \{(\kappa-1)^{3}\;\! \euet(L)\}\: \sigma^x \right) \\[10pt]
&+ (\es -\id) \left( -\,\es_{-1}\{3\eta(\kappa-1) \euet(L)
+\kappa(\kappa-1) \euka(L)\}\: \sigma^u \right)\\[10pt]
&+ (\es_{2} -\id)\left( \es_{-2}\{(\kappa-1)(\kappa_1-1) \euka(L)\}\: \sigma^u\right),
\end{align*}
where
\[
A_{\bkappa}=\frac{{\rm d} \kappa}{{\rm d}t}\,\es_{-1}\! \left(\frac{\partial L}{\partial \es \kappa} \right).
\]
Hence the terms that contribute to the conservation laws come from
\begin{equation}\label{BTexposEg}
A_{\mathcal{H}}=\mathcal{C}_x^0\, \sigma^x+\mathcal{C}_u^0\,\sigma^u+\mathcal{C}_u^1\,\es(\sigma^u),
\end{equation}
where
\begin{align*}
\mathcal{C}_x^0&= \es_{-1} \{(\kappa-1)^{3} \euet(L)\},\\[12pt]
\mathcal{C}_u^0&=-\,\es_{-1}\{3\eta(\kappa-1) \euet(L) +\kappa(\kappa-1)\euka(L)\} + \es_{-2}\{(\kappa-1)(\kappa_1-1) \euka(L)\},\\[12pt]
\mathcal{C}_u^1&=\es_{-1}\left\{(\kappa-1)(\kappa_1-1) \euka(L)\right\}.
\end{align*}
For this running example,
\[
\myAd(\rho_0)
=\left(\begin{array}{ccc} 1 & 0 & 0 \\ \ds\frac{3x_0}{(u_1-u_0)^{3}} & \ds\frac{1}{(u_1-u_0)^{3}} & 0 \\ \ds\frac{u_0}{u_1-u_0}  & 0 & \ds\frac{1}{u_1-u_0} \end{array}\right)
\]
and the invariantized form of the matrix of characteristics restricted to the variables $x_0$ and $u_0$ is
\[
\Phi_0(I)=\left(\begin{array}{c}\Phi^x_0\\ \Phi^u_0\end{array}\right)= \iup_0 \begin{pmatrix} x_0 & 1 & 0 \\ u_0 & 0 &1 \end{pmatrix}=\begin{pmatrix} 0 & 1 & 0 \\ 0 & 0 &1 \end{pmatrix}.
\]
Therefore, by \eqref{conlawsfinala}, the conservation laws are of the form $(\es -\id)A=0$, where
\[ \begin{aligned}
A&= \mathcal{C}_x^0 (0\ 1\ 0)\!\:\myAd(\rho_0) + \mathcal{C}_u^0 (0\ 0\ 1)\!\:\myAd(\rho_0)
+ \mathcal{C}_u^1 \es\big\{(0\ 0\ 1)\!\:\myAd(\rho_0)\big\}\\[2pt]
&= \left[\mathcal{C}_x^0(0\ 1\ 0) + \mathcal{C}_u^0 (0\ 0\ 1) + \mathcal{C}_u^1  (0\ 0\ 1)\,\myAd(\rho_1\rho_0^{-1})
\right]\myAd(\rho_0);
 \end{aligned}
\]
the last equality writes the conservation laws in the form of Equation (\ref{CLsgoodform}). It remains only to write the matrix $\myAd(\rho_1 \rho_0^{-1})$ in terms of $\eta$ and $\kappa$\:\!:
\[\myAd(\rho_1 \rho_0^{-1})=\myAd(K_0)=\iup_0(\myAd(\rho_1))=\left(\begin{array}{ccc} 1 & 0 & 0 \\ \ds\frac{3\eta}{(\kappa-1)^{3}} & \ds\frac{1}{(\kappa-1)^{3}} & 0 \\ \ds\frac{1}{\kappa-1}  & 0 & \ds\frac{1}{\kappa-1} \end{array}\right).
\] 
More generally, once one has solved for the frame, each $\myAd(\rho_j \rho_0^{-1})$ can be written in terms of $\bkappa$ and its shifts by using invariantization followed by the recurrence formula.
Doing the calculation, we find that $A=V(I)\myAd(\rho_0)$, where
\[
 V(I)=\left(\begin{array}{c}\es_{-1}\big\{(\kappa\!-\!1) \euka(L)\big\}\\[5pt]
\es_{-1} \big\{(\kappa\!-\!1)^{3}\;\! \euet(L)\big\}\\[5pt]
-\es_{-1}\big\{3\eta(\kappa\!-\!1) \euet(L)+(\kappa\!-\!1)^2\;\! \euka(L)\big\} +\es_{-2}\big\{(\kappa\!-\!1)(\kappa_1\!-\!1) \euka(L)\big\}
\end{array}\right)^{\!\!T}.
\]
For the particular Lagrangian \eqref{babylagrangian}, the solutions \eqref{kapsol}, \eqref{etasol} of the invariantized Euler--Lagrange equations yield
\begin{align}\label{vsol}
V_1&=-\,\tfrac{3}{2}\,\eta_{-1}(\kappa_{-1}\!-\!1)^{-3/2}=-\tfrac{3}{4}k_2\!\left[k_1\!+\!k_1^{-1}\!+\!(k_1\!-\!k_1^{-1})(2n\!-\!1)(-1)^n\right]\!+\tfrac{3}{2}k_3(-1)^n,\nonumber\\
V_2&=(\kappa_{-1}\!-\!1)^{3/2}=k_1^{3(-1)^{n+1}},\\
V_3&=-\tfrac{3}{2}\left[\eta(\kappa\!-\!1)^{-3/2}+\eta_{-1}(\kappa_{-1}\!-\!1)^{-3/2}\right]= -3k_2k_1^{(-1)^{n+1}}.\nonumber
\end{align}
In the coordinates we have used, the first element of $(\es -\id)A=0$ is the conservation law due to the scaling invariance, the second is due to invariance under translation of $x$, and the third is due to translation  
of $u$.
\end{example}

\begin{table}
	\centering{
		\begin{tabular}{|p{0.5cm}|p{0.4cm}| p{9cm} |}
			\hline
			\vspace*{-6pt}$\phi^x$&\vspace*{-6pt}$\phi^u$&\vspace*{-6pt}Invariantized first integral: $V(I)=\iup_0\{A_\mbu(n,\bphi)\}$\\[2pt]
			\hline 
			$3x$&$u$&$\phantom{\Bigg\vert}\iup_0\left(\es_{-1}\ds\frac{\partial\:\! \upL}{\partial u_{2}}\right)$\\[15pt]
			\hline
			1&0&$\phantom{\Bigg\vert}\iup_0\left(\es_{-1}\ds\frac{\partial\:\! \upL}{\partial x_{1}}\right)$
			\\ [12pt]\hline
			0&1&$\phantom{\Bigg\vert}\iup_0\left(\es_{-1}\ds\frac{\partial\:\! \upL}{\partial u_{2}}-\ds\frac{\partial\:\! \upL}{\partial u_{0}}\right)$\\ \hline
		\end{tabular}
	\caption{Infinitesimals and invariantized first integrals for the Lagrangian (\ref{babylagrangian}).}\label{invfi}}
\end{table}
\smallskip

\begin{remark}\textbf{Calculation of the conservation laws}
	There is another way to calculate the laws for difference frames. By 
	Corollary \ref{goodCLs}, one can use symbolic software to calculate the conservation laws in the original variables, 
	and then use the Replacement  Rule, Theorem \ref{reprule},
	to obtain the invariantized first integrals $ V(I)$ (see Table \ref{invfi}, {which is obtained from Table \ref{tableEG} by applying $\iota_0$, taking the normalisation equations into account}). This is because the Replacement Rule sends $\rho_0$ to the identity matrix. 
	The recurrence formulae can then be used to write $ V(I)$ in terms of the generating invariants. 
	This means that the methods to solve for the extremals in the original variables, given in the next section, can still be used without having to 
	perform the more complex, invariantized summation by parts computation.
\end{remark}


\section{Solving for the original dependent variables $\mbu_0$, once the generating invariants are known}\label{SolveSec}

In this section we show how to find the solutions $\mbu_0$ to the original Euler--Lagrange equations, once the invariant Euler--Lagrange equations have been solved for the generating invariants $\kappa^\alpha$. The starting-point is that $\bkappa$ is a known function of $n$ and some arbitrary constants (which are determined if initial data are specified). There are three methods, depending on the information available. We use the running expository example to illustrate each method.

\subsection{How to solve for $\mbu_0$ from the invariants, knowing only the Maurer--Cartan matrix.}
This method can be used for any invariant difference system. Indeed, when the adjoint representation of the Lie group is trivial, it is the only available method. 
 
Assume that the Maurer--Cartan matrix $ K_0=\rho_{1}\rho_0^{-1}$ is known in terms of the generating invariants, so that it can be written in terms of $n$ (and some arbitary constants). This yields a system of recurrence relations for $\rho_0$, namely
\begin{equation}\label{recrelsrhoparams}\rho_{1} = K_0 \rho_0.\end{equation}

\begin{definition} The system (\ref{recrelsrhoparams}) is known as the set of \emph{Maurer--Cartan equations} for the frame $\rho$.
 \end{definition}

Once the Maurer--Cartan equations for $\rho_0$ have been solved, one can obtain $\mbu_0$ from
\begin{equation}\label{getufromrho} u_0^\alpha = \rho_0^{-1} \left( \rho_0\cdot u_0^\alpha\right)=\rho_0^{-1} I^{\alpha}_{0,0}\,;\end{equation}
the invariant $I^{\alpha}_{0,0}$ is known, either from the normalization equations or from the set of generating invariants already determined.
\begin{example}(Example \ref{simpleExOne} cont.)\  
From Equation (\ref{examplebabyframeA:C}), the Maurer--Cartan matrix is
\[
K_0=\left(\begin{array}{ccc} (\kappa-1)^{-3} & 0 & -\eta (\kappa-1)^{-3}\\ 0 & (\kappa -1)^{-1} & - (\kappa -1)^{-1}\\ 0&0&1\end{array}\right).
\]
Hence, setting $\lambda_k$, $a_k$ and $b_k$ to be the parameter values for the group element $\rho_k$, the set of Maurer--Cartan equations is
$$\left(\begin{array}{ccc} \lambda_{1}^3 &0&a_{1}\\ 0&\lambda_{1} & b_{1}\\0&0&1\end{array}\right) 
=K_0
\left(\begin{array}{ccc} \lambda_0^3 &0&a_0\\ 0&\lambda_0 & b_0\\0&0&1\end{array}\right),$$
which amount to three recurrence relations for the group parameters:
\begin{equation}\label{exabeqnssec7}\begin{array}{rcl}
\lambda_{1} &=& (\kappa-1)^{-1}\lambda_0,\\
a_{1}&=&(\kappa-1)^{-3} \left(a_0 -\eta\right),\\
b_{1}&=&(\kappa-1)^{-1}\left(b_0-1\right).
\end{array}\end{equation}
Now suppose that we know the general solution of these recurrence relations. The normalization equations give $\rho_0\cdot x_0=0$ and $\rho_0\cdot u_0=0$, so
$$\begin{array}{rclll}
x_0&=& \rho_0^{-1}\cdot (\rho_0\cdot x_0) &= \lambda_0^{-3}(\rho_0\cdot x_0 -a_0) &= -\lambda_0^{-3}{a_0},\\
u_0&=&\rho_0^{-1}\cdot (\rho_0\cdot u_0) &= \lambda_0^{-1}(\rho_0\cdot u_0 -b_0) &= -\lambda_0^{-1}{b_0}.
\end{array}$$
\end{example}

\subsection{Solving for $\mbu_0$ from the invariants and conservation laws when the adjoint representation is nontrivial}\label{methodforEulerElas}
This method works when the adjoint representation is not the identity representation.
The conservation laws give
\begin{equation}\label{onedclsolve}V(I)\!\:\myAd(\rho_0) =c\end{equation}
where $c$ is a constant row vector. The components $V_i$ depend only on $\bkappa$, and are therefore known functions of $n$.
As $\myAd(g)$ is known in terms of the group parameters, Equation (\ref{onedclsolve}) yields equations for these parameters.

If the adjoint action of the group on its Lie algebra is not transitive, the algebraic system of equations for the parameters may be under-determined. To complete the solution, it is then necessary to supplement this system with the Maurer-Cartan equations (\ref{recrelsrhoparams}). 
Even so, the algebraic equations coming from the conservation laws can ease, considerably, the problem of solving the Maurer--Cartan equations alone.
Once $\rho_0$ is known as a function of $n$, Equation (\ref{getufromrho}) yields $\mbu_0$, as before.

\begin{example}(Example \ref{simpleExOne} cont.)\  
For the running example, (\ref{onedclsolve}) is
$$(V_1\ V_2\ V_3)\left(\begin{array}{ccc} 1& 0 & 0\\ -3 a_0 & \lambda_0^3 & 0\\ -b_0&0&\lambda_0\end{array}\right)=(c_1\ c_2\ c_3).$$
We obtain immediately
$\lambda_0=c_3/V_3$ and hence, a first integral of the Euler--Lagrange equations:
 \begin{equation}\label{firstintrunexsec7}\frac{V_2}{(V_3)^3}=\frac{c_2}{(c_3)^3}\,.\end{equation}
The remaining equation is a linear expression for $a_0$ and $b_0$,
\begin{equation}\label{sec7part2third} 
3a_0V_2+b_0V_3-V_1+c_1=0.
\end{equation}
If one of the second and third equations of (\ref{exabeqnssec7}) can be solved, \eqref{sec7part2third} yields the remaining parameter.
\end{example}

\subsection{Solving for  $\mbu_0$ from $\bkappa$ from the conservation laws, and with a nontrivial adjoint representation of $\rho$ which is known as a function of $\mbu_0$}
In this case we consider the conservation laws
$V(I)\!\:\myAd(\rho_0)=c$, taking into account that $\rho_0(\mbu)$ is known as a function of the dependent variables. 
One can sometimes derive explicit equations for $\mbu$ which are simple to solve. We illustrate the possibilities in the running example.
\begin{example}(Example \ref{simpleExOne} cont.)\  
The conservation laws amount to
\begin{equation}\label{case3sec7mxeqn}(V_1\ V_2\ V_3)\left(\begin{array}{ccc} 1& 0 & 0\\[6pt] \ds\frac{3x_0}{(u_{1}-u_0)^{3}} & \ds\frac{1}{(u_{1}-u_0)^{3}} & 0\\[12pt] \ds\frac{u_0}{u_{1}-u_0}&0&\ds\frac{1}{u_{1}-u_0}\end{array}\right)=(c_1\ c_2\ c_3).
\end{equation}
We obtain once more the first integral (\ref{firstintrunexsec7})  and the simple recurrence relation
\begin{equation}\label{urr}
u_{1}-u_0=V_3/c_3.
\end{equation}
Once this is solved for $u_0$, one can obtain $x_0$ from the first column of (\ref{case3sec7mxeqn}).

For the Lagrangian \eqref{babylagrangian}, each $V_r$ is given \eqref{vsol} in terms of $n$ and $k_i,\ i=1,2,3$. The first integral \eqref{firstintrunexsec7} yields $c_3=-3k_2\:\!c_2^{1/3}$. Assuming that $k_2$ is nonzero, it is convenient to define $k_4=c_2^{-1/3}$; then the general solution of \eqref{urr} is
\[
u_0=\tfrac{1}{4}k_4\left[2(k_1\!+\!k_1^{-1})n+(k_1\!-\!k_1^{-1})(-1)^n+k_5\right],
\]
where $k_5$ is an arbitrary constant. Finally, the first column of \eqref{case3sec7mxeqn} gives
\[
x_0=k_4^3\left[k_2nk_1^{(-1)^n}-\tfrac{1}{2}k_3(-1)^n+k_6\right],
\]
where $k_6=c_1/3+k_2(k_1\!+\!k_1^{-1}\!+\!k_5)/4$ is the remaining arbitrary constant.
\end{example}


\section{Lagrangians invariant up to a divergence}\label{Divgen}

So far, we have considered only Lagrangians $\upL$ that are invariant under a Lie group $G$ of variational symmetries, that is, $g\cdot\upL=\upL$ for all $g\in G$. However, Noether's theorem merely requires the action $\mathcal{L}$ to be invariant. This broader definition of a Lie group of variational symmetries requires that for each $g\in G$, there exists a function $P_g$ of $n$ and a finite number of shifts of $\mbu_0$ such that
\[
g\cdot\upL=\upL+(\es-\id)P_g\,.
\]
Without loss of generality, we set $P_e=0$.
This useful generalization can be treated in the invariant framework by introducing a new dependent variable $\zeta$ such that
\begin{equation}\label{zetact}
g\cdot\zeta_j=\es_j\big(\zeta_0-P_g\big),\qquad  j\in\mathbb{Z}.
\end{equation}

\begin{lemma}
	With the above notation, the group action on $(\es-\id)\zeta$ defined by \eqref{zetact} is a left action. Furthermore, the modified Lagrangian
	\begin{equation}\label{lbar}
	\overline{\upL}=\upL+(\es-\id)\zeta_0
	\end{equation}
	is invariant under all $g\in G$.
\end{lemma}
\begin{proof}
	Let $g,h\in G$. By definition, taking the left action of $G$ on $\mbu$ into account,
	\[
	(\es-\id)P_{hg}=(hg)\cdot\upL -\upL=h\cdot(\upL+(\es-\id)P_g)-\upL=(\es-\id)(P_h+h\cdot P_g),
	\]
	where the last equality is a consequence of the prolongation formula.
	Therefore
	\[
	(hg)\cdot\big((\es-\id)\zeta_0\big)=(\es-\id)(\zeta_0-P_{hg})=(\es-\id)(\zeta_0-P_h-h\cdot P_g)=h\cdot\big(g\cdot\big((\es-\id)\zeta_0\big)\big),
	\]
	which extends to a left action on all $(\es-\id)\zeta_j$ by the prolongation formula. Invariance of the modified Lagrangian follows:
	\[
	g\cdot\overline{\upL}=g\cdot\upL+g\cdot(\zeta_1-\zeta_0)=\upL+(\es-\id)P_g+(\es-\id)(\zeta_0-P_g)=\overline{\upL}\,.\qedhere
	\]
\end{proof}

As $\overline{\upL}$ and $\upL$ differ by a divergence, they yield the same Euler--Lagrange equations. Therefore $\overline{\upL}$ can be used to obtain the conservation laws by the difference moving frame method. Only the original dependent variables can be used for normalization, because $\zeta$ does not appear in the Euler--Lagrange equations. Moreover, $\zeta$ does not appear in $A_{\mathcal{H}}$ for the modified Lagrangian, so the expression for the conservation laws in terms of invariants and $\myAd(\rho_0)$ is unaffected by the modification. 

An example appears in the next section, \S\ref{Exgen}.

\section{Explicit dependence on $n$ changes nothing}\label{Exgen}

So far, we have considered examples whose Lagrangian and variational symmetries do not depend explicitly on $n$. However, the difference frame reduction takes place on a prolongation space over a single (arbitrary) base point $n$. Consequently, $n$ should be regarded in the calculations  as a parameter.

Consider the Lagrangian  
\[
\upL=\frac{(v_2-v_0)(v_3-v_1)}{(u_0-v_0)(u_1-v_1)}+2\ln(u_0-v_0).
\]
While this Lagrangian does not depend explicitly on $n$, there is a three-parameter Lie group of variational symmetries depending explicitly on $n$, $g\cdot(u,v)=(\tilde{u},\tilde{v})$, where
\begin{align}\label{twistact}
\tilde{u}&= \exp\left\{(-1)^n a_1\right\} u + a_2 + a_3(-1)^n=u\cosh a_1+a_2+(-1)^n\big\{u\sinh a_1+a_3\big\},\nonumber\\
\tilde{v}&= \exp\left\{(-1)^n a_1\right\} v + a_2 + a_3(-1)^n=v\cosh a_1+a_2+(-1)^n\big\{v\sinh a_1+a_3\big\}.
\end{align}
The infinitesimal generators $\mbv_r$ corresponding to $a_r$ are
\[
\mbv_1=(-1)^nu\,\partial_{u}+(-1)^nv\,\partial_{v},\qquad \mbv_2=\partial_{u}+\partial_{v},\qquad \mbv_3=(-1)^n\partial_{u}+(-1)^n\partial_{v}.
\]
So
\begin{equation}\label{Adex3}
(\mbv_1\ \cdots\ \mbv_R)=(\widetilde\mbv_1\ \cdots\ \widetilde\mbv_R)\!\;\myAd(g),\ \text{where}\ \myAd(g)=\begin{pmatrix}1 & 0  & 0\\ -a_3 & \cosh a_1&\sinh a_1\\ -a_2&\sinh a_1&\cosh a_1\end{pmatrix}\!.
\end{equation}
For this group action, the standard representation is not faithful, so we shall use the Adjoint representation, which is faithful.

In this example, $\upL$ is not invariant:
\[
g\cdot\upL=\upL+2a_1(-1)^n=\upL+(\es-\id)\left\{a_1(-1)^{n+1}\right\}.
\]
Accordingly, by the result in \S\ref{Divgen}, we define a new dependent variable $\zeta$ and the (invariant) modified Lagrangian $\overline{\upL}(u_0,v_0,\zeta_0,u_1,v_1,\zeta_1,v_2,v_3)$ such that
\[
\widetilde{\zeta_0}=\zeta_0+a_1(-1)^n,\qquad \widetilde{\zeta_1}=\zeta_1+a_1(-1)^{n+1},\qquad \overline{\upL}=\upL+(\es-\id)\zeta_0\,;
\]
here $\widetilde{\zeta_j}$ denotes $g\cdot\zeta_j$. The extension of the infinitesimal generators $\mbv_r$ to include their action on $\zeta$ is as follows:
\[
\mbv_1=(-1)^nu\,\partial_{u}+(-1)^nv\,\partial_{v}+(-1)^n\partial_\zeta,\quad \mbv_2=\partial_{u}+\partial_{v},\quad \mbv_3=(-1)^n\partial_{u}+(-1)^n\partial_{v}.
\]
These generators satisfy the same commutation relations as their unextended counterparts, so \eqref{Adex3} is unchanged.

We choose the moving frame $\rho_0$ defined by the normalization
\[
\rho_0\cdot u_{0}=1,\qquad \rho_0\cdot v_{0}=0,\qquad \rho_0\cdot v_{1}=0,
\]
which amounts to
\begin{align}\label{param}
a_1&=(-1)^{n+1}\ln(u_0-v_0),\nonumber\\ a_2&=-\tfrac{1}{2}\left[v_1(u_0-v_0)+v_0/(u_0-v_0)\right],\\ a_3&=\tfrac{1}{2}(-1)^n\left[v_1(u_0-v_0)-v_0/(u_0-v_0)\right]\nonumber.
\end{align}
Consequently, the frame (in the Adjoint representation) is
\[
\rho_0=\myAd(\rho_0)=\begin{pmatrix}1 & 0  & 0\\ -a_3 & \cosh a_1&\sinh a_1\\ -a_2&\sinh a_1&\cosh a_1,\end{pmatrix}
\]
where $a_1, a_2$ and $a_3$ are given by \eqref{param}; hence
\begin{equation}\label{sca1}
\cosh a_1=\frac{1}{2}\left[u_0-v_0+\frac{1}{u_0-v_0}\right],\quad \sinh a_1=\frac{(-1)^{n+1}}{2}\left[u_0-v_0-\frac{1}{u_0-v_0}\right].
\end{equation}
The fundamental {{difference}} invariants are
\begin{align*}
\kappa&=\rho_0\cdot u_1=(u_0-v_0)(u_1-v_1),\\
\mu&=\rho_0\cdot v_2=(v_2-v_0)/(u_0-v_0),\\
\nu&=\rho_0\cdot\zeta_0=\zeta_0-\ln(u_0-v_0).
\end{align*}
Therefore, in terms of these invariants, $\overline{\upL}$ amounts to
\begin{equation}\label{ex3lag}
L=\mu\mu_1+\ln\kappa+\nu_1-\nu_0,
\end{equation}
and 
\begin{equation}\label{TwistEgKinv}
\myAd(K_0) =\iup_0(\myAd(\rho_1))= \left(\begin{array}{ccc}
1 &0 & 0\\[12pt]
\frac{(-1)^n}{2}\,\kappa\mu& \frac12\left( \kappa +\kappa^{-1}\right)&\frac{(-1)^n}{2}\left( \kappa - \kappa^{-1}\right)\\[12pt]
\frac12\:\!\kappa\mu& \frac{(-1)^n}{2}\left( \kappa - \kappa^{-1}\right)&\frac12\left( \kappa +\kappa^{-1}\right)
\end{array}
\right).
\end{equation}

We now construct the differential--difference syzygies. Noting that
\[
\iup_0(u_j')=(u_0-v_0)^{(-1)^{j-1}}u_j',\qquad\iup_0(v_j')=(u_0-v_0)^{(-1)^{j-1}}v_j',\qquad\iup_0(\zeta_j')=\zeta_j',
\]
the fundamental differential--difference invariants are
\begin{equation}\label{TwistEeDiffInvs}
\sigma^u =\iup_0(u_0') =  \frac{u_0'}{u_0-v_0}\,,\qquad  \sigma^v =\iup_0(v_0') = \frac{u_0'}{u_0-v_0}\,,\qquad \sigma^\zeta=\zeta_0'.
\end{equation}
Consequently,
\begin{equation}\label{TwistEgdiffInvRR}
\iup_0(u_1')=\kappa\,\es\:\!\sigma^u,\quad\ \iup_0(v_1')=\kappa\,\es\:\!\sigma^v,\quad\ 
\iup_0(v_2')=\kappa^{-1}\kappa_1\,\es_2\sigma^v,
\quad\iup_0 (\zeta_1')=\es\:\!\sigma^\zeta.
\end{equation}
and thus the differential--difference syzygies needed to calculate the Euler--Lagrange equations are
\begin{equation}\label{TwistEqHeqn1}
\begin{array}{rcl}
\ds\frac{{\rm d}}{{\rm d}t} \left(\begin{array}{c} \kappa  \\ \mu\\ \nu \end{array}\right)
&=& \left(\begin{array}{ccc} \mathcal{H}_{11} & \mathcal{H}_{12}&0\\ \mathcal{H}_{21} & \mathcal{H}_{22}&0\\
-1& 1&1\end{array}\right)
\left(\begin{array}{c} \sigma^u\\ \sigma^v\\ \sigma^\zeta \end{array}\right)
\end{array}
\end{equation}
where
\begin{equation}\label{TwistEqHeqn2}
\mathcal{H}_{11}= \kappa(\es+\id),\quad \mathcal{H}_{12}=-\mathcal{H}_{11},\quad
\mathcal{H}_{21}=  -\mu\,\id,\quad
\mathcal{H}_{22}=\kappa^{-1}\kappa_1 \es_2+(\mu-\!1)\:\!\id.
\end{equation}
Therefore
\begin{align}\label{ex3var}
\frac{\upd L}{\upd t}=&\{\mathcal{H}_{11}^*\euka(L)+\mathcal{H}_{21}^*\eumu(L)-\eunu(L)\}\,\sigma^u+\{\mathcal{H}_{12}^*\euka(L)+\mathcal{H}_{22}^*\eumu(L)+\eunu(L)\}\,\sigma^v\nonumber\\
&+\eunu(L)\,\sigma^\zeta+(\es-\id)(A_{\bkappa}+A_{\mathcal{H}}),
\end{align}
where
\[
A_{\mathcal{H}}=\left[\es_{-1}(\kappa\euka(L))\right]\!(\sigma^u-\sigma^v)+(\es+\id)\left\{\es_{-2}\!\left[\kappa^{-1}\kappa_1\eumu(L)\right]\sigma^v\right\}
\]
and, for the particular Lagrangian \eqref{ex3lag}, $A_{\bkappa}=\mu_{-1}\mu'+\nu\!\phantom{.}'$.

The invariantized Euler--Lagrange equations are obtained from the coefficients of $\sigma^u,\sigma^u$ and $\sigma^\zeta$ in \eqref{ex3var}, ignoring terms in $A_{\bkappa}$ and $A_{\mathcal{H}}$. By construction, the coefficient of $\sigma^\zeta$ gives $\eunu(L)=0$; the remaining Euler--Lagrange equations simplify to
\begin{align}
0&=(\es_{-1}+\id)\{\kappa\euka(L)\}-\mu\eumu(L)=2-\mu(\mu_{-1}+\mu_1),\label{ex3eu}\\
0&=\es_{-2}\{\kappa^{-1}\kappa_1\eumu(L)\}-\eumu(L)=\left((\kappa_{-2})^{-1}\kappa_{-1}\es_{-2}-\id\right)(\mu_{-1}+\mu_1).\label{ex3ev}
\end{align}
The general (real-valued) solution of \eqref{ex3eu} is
\begin{equation}\label{ex3musol}
\mu=k_2^{(-1)^n}\left(1-(k_1)^2\right)^{-\lfloor n/2\rfloor}\left(1+k_1^{(-1)^{n+1}}\right)^n,\qquad k_1\neq\pm1,\ k_2\neq 0.
\end{equation}
To complete the solution, the following identities are useful:
\[
\mu_2=\frac{1+k_1^{(-1)^{n+1}}}{1+k_1^{(-1)^n}}\,\mu,\qquad \mu\mu_1=1+k_1(-1)^n.
\]
The first of these enables \eqref{ex3ev} to be solved:
\begin{equation}\label{ex3kapsol}
\kappa=c\left(1-(k_1)^2\right)^{-1} \left(1+k_1^{(-1)^{n+1}}\right),		
\end{equation}
where the (non-zero) constant $c$ will be determined later.

The conservation laws come from $A_{\mathcal{H}}$ which, for the Lagrangian \eqref{ex3lag}, amounts to 
\begin{equation}\label{TwistAHterms}
A_{\mathcal{H}}=\sigma^u-\sigma^v+(\es+\id)\left\{(\kappa_{-2})^{-1}\kappa_{-1}(\mu_{-3}+\mu_{-1})\,\sigma^v\right\}.
\end{equation}
Note: by construction, $\zeta$ and its invariantization do not contribute to the conservation laws; $\nu\!\phantom{.}'$ appears in $A_{\bkappa}$, but not in $A_{\mathcal{H}}$. 

One can complete the solution of the problem using \eqref{TwistAHterms}, but it is better to use the Euler--Lagrange equations \eqref{ex3eu} and \eqref{ex3ev} to simplify $A_{\mathcal{H}}$ first; this gives equivalent conservation laws with
\[
A_{\mathcal{H}}=\sigma^u-\sigma^v+(\es+\id)\left(2\mu^{-1}\sigma^v\right).
\]

To calculate the replacements for $\sigma^u$ and $\sigma^v$, we need the matrix of characteristics, restricted to the original dependent variables:
\begin{equation}\label{OmegTwEg}
\Phi(u_0,v_0)=\bordermatrix{ & a_1 & a_2 & a_3\cr
	u_0 & (-1)^n u_0 & 1 & (-1)^n\cr
	v_0 & (-1)^n v_0& 1& (-1)^n\cr}.
\end{equation}
The invariantization of this matrix gives the replacements
\begin{equation}\label{omegaReplacTwEg1}
\sigma^u \mapsto \left((-1)^n\ 1\ (-1)^n\right)\myAd(\rho_0),\qquad 
\sigma^v \mapsto \left(0\ 1\ (-1)^n\right)\myAd(\rho_0).
\end{equation}
The replacement for $\es\sigma^v$ is
\begin{equation}\label{omegaReplacTwEg3}
\es\sigma^v \mapsto \left(0\ 1\ (-1)^{n+1}\right)\myAd(\rho_1)=\left(0\ 1\ (-1)^{n+1}\right)\myAd(K_0)\myAd(\rho_0).
\end{equation}
Collecting terms, the conservation laws are 
\begin{equation}\label{TwistEgCLs}
(c_1\ c_2\ c_3)=\left((-1)^n\quad\ \frac{2}{\mu}+\!\frac{2}{\kappa\mu_1}\quad\ \frac{2(-1)^n}{\mu}+\!\frac{2(-1)^{n+1}}{\kappa\mu_1}
\right)\myAd(\rho_0).
\end{equation}
Bearing in mind that $\kappa=(u_0-v_0)(u_1-v_1)$, \eqref{TwistEgCLs} amounts to
\begin{align}
c_1&=(-1)^n\left[1-(\es-\id)\left\{\frac{2v_0}{\mu(u_0-v_0)}\right\}\right],\label{ex3c1}\\[5pt]
c_2&=(\es+\id)\left\{\frac{2}{\mu(u_0-v_0)}\right\},\quad
c_3=(\es+\id)\left\{\frac{2(-1)^n}{\mu(u_0-v_0)}\right\}.\label{ex3c23}
\end{align}
The general solution of \eqref{ex3c23} is
\begin{equation}\label{ex3uvsol}
u_0-v_0=\frac{4}{\mu(c_2+c_3(-1)^n)}\,,\qquad c_2^2-c_3^2\neq 0,
\end{equation}
where $\mu$ is given by \eqref{ex3musol}. Therefore, the value of the undetermined constant in \eqref{ex3kapsol} is $c=16/(c_2^2-c_3^2)$ and the general solution of \eqref{ex3c1} is
\begin{equation}\label{ex3vsol}
v_0=\frac{2n+c_1(-1)^n+k_3}{\mu(c_2+c_3(-1)^n)}\,,
\end{equation}
where $k_3$ is an arbitrary constant. This yields $u_0$ from \eqref{ex3uvsol}, completing the solution of the Euler--Lagrange equations in the original variables.


\section{Application to the study of Euler's elastica}\label{Elastica}
As a final example, we study a discrete variational problem analogous to that of the smooth Euler elastica, 
\begin{equation}
\label{lagran}
\mathcal{L}=\int \kappa^2 {\rm d}s,\qquad \kappa=\frac{u_{xx}}{\left(1+u_x^2\right)^{3/2}},\qquad {\rm d}s=\sqrt{1+u_x^2}\,{\rm d}x,
\end{equation}
where $\kappa$ is the Euclidean curvature and $s$ is the Euclidean arc length. Converting to derivatives with respect to arc length $s$, the Euler--Lagrange equation is
$$\kappa_{ss}+\frac12\kappa^3=0,$$
for which a first integral was found by Euler in his masterpiece \citet{Euler} (see also \citet{levien} for a mathematical history of the problem).

{{The  aim is to design  the discrete  Lagrangian in such a way that not only the discrete Euler--Lagrange equations, but also all the discrete conservation laws  become, in an appropriate continuum limit, the smooth Euler--Lagrange equations and conservation laws of a variational problem.  Our method involves taking a difference frame for $SE(2)$ which has for its continuum limit the smooth $SE(2)$ frame, see 
\citet{Mansfield:2010aa,GonMan2}; by matching this heart of the two calculations, one for the smooth and one for the difference variational problem, we match not only the Euler--Lagrange equations but {all three} conservation laws, as all the relevant formulae align allowing convergence to be proven readily.}}

 {{By contrast, Ge's famous no-go theorem (see \citet{GeMars})  states that a symplectic integrator, possibly after reduction so that only the conservation of energy remains, cannot exactly preserve the smooth energy without computing the exact solution. }}
 Conservation of energy in the smooth cases arises when a Lagrangian is invariant
under translations in the independent variable. When computing a difference analogue, the independent variable must appear as a discrete dependent variable and the difference Lagrangian must be invariant under translation in this, so that conservation of energy in the smooth case becomes
a conservation of a linear momentum in the difference analogue. For our example here, it is translation in $x$ in the smooth case which is incorporated.

{{That our method of works in general is an open conjecture. To evidence this conjecture, we calculate all the relevant quantities in detail.}}

The Euclidean group of rotations and translations in the plane acts on curves $(x, u(x))$ as 
\[
\left( 
\begin{array}{c}
x\\u \end{array}
\right)
\mapsto
R_{\theta}\left( 
\begin{array}{c}
x\\u \end{array}
\right) +\left( 
\begin{array}{c}
a\\b \end{array}
\right)=\left( 
\begin{array}{c}
\tilde{x}\\ \tilde{u} \end{array}
\right),\qquad R_\theta=\left( 
\begin{array}{rr}
\cos\theta&-\sin\theta\\ \sin\theta&\cos\theta \end{array}
\right).
\]
Choosing the normalization equations to be
\begin{equation}\label{smoothNeqs}
\tilde{x}=0,\qquad \tilde{u}=0,\qquad \widetilde{u_x}=0,
\end{equation}
 we obtain a smooth frame, denoted by $\widehat{\rho}$, namely
\[
\widehat{\rho}=\left(\begin{array}{cc} R_{\theta} & -R_{\theta} \left(\begin{array}{c} x\\u\end{array}\right)\\0&1\end{array}\right),
\]
where $R_{\theta}$ is the rotation matrix with $\sin\theta = -u_x/\sqrt{1+u_x^2}$, $\cos\theta = 1/\sqrt{1+u_x^2}$. This frame  satisfies
\begin{equation}
\label{maurer1}
\widehat{\rho}_s \widehat{\rho}^{\:-1}= \frac{1}{\sqrt{1+u_x^2}}\, \widehat{\rho}_x \widehat{\rho}^{\:-1}=\left(\begin{array}{ccc} 0 & \kappa & -1 \\ -\kappa & 0 &0 \\ 0 & 0 & 0\end{array} \right).
\end{equation}
With this frame, the conservation laws for the Lagrangian (\ref{lagran}) are, in terms of the 
moving frame $\widehat{\rho}$  
\citep[see][]{GonMan2,Mansfield:2010aa}:
$$(-\kappa^2\ -\!2\kappa_s\,\ 2\kappa)\underbrace{
\left(\begin{array}{ccc} x_s & u_s & xu_s-ux_s\\ -u_s & x_s & xx_s+uu_s\\ 0&0&1\end{array}\right)}_{\myAd(\widehat{\rho}\,)}
=( c_1\ c_2\ c_3).$$
Using the identity $x_s^2+u_s^2=1$, this amounts to
\begin{equation}\label{EEconlaws}(-\kappa^2\ -\!2\kappa_s\,\ 2\kappa)=( c_1\ c_2\ c_3)\left(\begin{array}{ccc} x_s & -u_s & u\\ u_s & x_s & -x\\ 0&0&1\end{array}\right).\end{equation}
The same identity gives a first integral for the Euler--Lagrange equation,
\begin{equation}\label{EEfirstint} \kappa^4+4\kappa_s^2=c_1^2+c_2^2.\end{equation}
Eliminating $x_s$ from the first two columns of \eqref{EEconlaws} gives
 \begin{equation}\label{EExseq} u_s=\frac{1}{c_1^2+c_2^2}\, \left( 2 c_1 \kappa_s-c_2\kappa^2\right).\end{equation}
By solving (\ref{EEfirstint}), (\ref{EExseq}) and the third column of (\ref{EEconlaws}) (to determine $x$), we obtain the smooth solution in Figure \ref{EulerSols}, {once the constants of integration $c_1$ and $c_2$ are determined}.

The idea is to take a difference frame with matching normalization equations,  and to take the discrete analogues of curvature and arc length to be those playing the same role when compared to the smooth
Maurer--Cartan invariants given in Equation (\ref{maurer1}). We now explain these remarks.
Consider the action of $\es E(2)$ in the plane where the points $\mbu_j$  have coordinates $(x_j,u_j)$,
and take the frame
\[
\rho_{0}=
\left(
\begin{array}{cc}
R_{\theta_{0}} & -R_{\theta_{0}}\left(\begin{array}{c} x_0\\u_0\end{array}\right) \\
0 & 1
\end{array}
\right),
\]
(using the standard representation) such that  the normalization equations are $$\rho_0\cdot \mbu_0=(0,0), \qquad \rho_0\cdot (\mbu_1-\mbu_0)=(*,0).$$ In other words, $R_{\theta_{0}}$ is the rotation matrix that sends $ \mbu_{1}- \mbu_0$ to a row vector with a zero second component, so that
$\sin \theta_0=-(u_{1}-u_0)/\ell$ and $\cos\theta_0=(x_{1}-x_0)/\ell$, where $\ell=| \mbu_{1}- \mbu_0|$. These discrete normalization equations match, in some sense,  the 
normalization equations (\ref{smoothNeqs}) for the smooth frame. Then 
\[
K_0=\rho_{1}\rho_0^{-1}=\left(\begin{array}{cc} R_{\hth} & -R_{\theta_{1}}\left(\begin{array}{c} x_1-x_0\\u_1-u_0\end{array}\right)\\0&1\end{array}\right)
=\left(\begin{array}{cc} R_{\hth} & -R_{\hth}\left(\begin{array}{c} \ell\\0\end{array}\right)\\0&1\end{array}\right),
\]
where $\hth=\theta_{1}-\theta_{0}$. Therefore the generating invariants are $\hth$ and $\ell$.

In order to see the discrete analogues of curvature and arc length, we consider $\widehat{\rho}_x\widehat{\rho}^{\:-1}$ to be approximated by 
\[\left(\widehat{\rho}(x+h_x)-\widehat{\rho}(x)\right)\widehat{\rho}(x)^{\:-1}/h_x=\left(\widehat{\rho}(x+h_x)\widehat{\rho}(x)^{\:-1}-\mbox{Id}\right)/h_x,\] where $\text{Id}$ is the identity matrix, and
$\widehat{\rho}(x+h_x)\widehat{\rho}(x)^{\:-1}$ to be approximated by $K_0$ when $x=x_0$ and $h_x=x_{1}-x_0$.

Observing that the component of the first row and second column of the matrix $K_{0}-\mbox{Id}$ is $-\sin\hth$ and that, to first order in $\hth$, the component of the first row and third column of the matrix 
$K_{0}-\mbox{Id}$ is $-\ell$,
we take the discrete analogue of ${\rm d}s$ to be $\ell$ and the discrete analogue of $\kappa$ to be
$$\overline{\kappa}=-\ell^{-1}\sin\hth.$$
Hence, we consider the variational problem
\[
\mathcal{L}=\sum \ell^{-1}\sin^2\hth,
\]
with $\ell$ and $\hth$ as the fundamental invariants. The differential--difference syzygies are
\begin{align*}
\ell\!\!\:\phantom{|}'&=\cos\hth\,\es\sigma^x+\sin\hth\,\es\sigma^u-\sigma^x,\\
\hth'&=(\es-\id)\left(\ell^{-1}[\sin\hth\,\es\sigma^x-\cos\hth\,\es\sigma^u+\sigma^u]\right),
\end{align*}
where $\sigma^x=\iup_0(x_0')$ and $\sigma^u=\iup_0(u_0')$.
Applying the theory developed in this paper, the invariantized Euler--Lagrange equations are
 \[
 \begin{array}{l}
\{\es_{-1}(\cos\hth)\,\es_{-1}-\id\}\eul (L) + \left\{\es_{-1}\!\left(\ell^{-1}\sin\hth\right)(\es_{-2}\!-\!\es_{-1})\right\}\eut (L)=0,\\[5pt]
\{\es_{-1}(\sin{\hth})\,\es_{-1}\}\eul (L) +\left\{\ell^{-1}(\es_{-1}\!-\!\id)-\es_{-1}\!\left(\ell^{-1}\cos{\hth}\right)(\es_{-2}\!-\!\es_{-1})\right\}\eut (L)=0,
\end{array}
\]
where
\[
\eut (L)=\frac{\partial L}{\partial \hth}=\ell^{-1}\sin(2\hth), \qquad \eul (L)=\frac{\partial L}{\partial \ell}=-\ell^{-2}\sin^2\hth.
\]
These equations are then solved for $\ell$ and $\hth$.
We note that (the shifts of) these Euler--Lagrange equations can be written in the form,
\[ \left(\begin{array}{cc} \cos\hth & -\sin \hth\\ \sin \hth & \cos\hth\end{array}\right) \left(\begin{array}{c} \ell^{-1}\left(\es_{-1} -\id\right) \eut(L)\\ \eul\end{array}\right) = \es \left(\begin{array}{c}  \ell^{-1}\left(\es_{-1} -\id\right) \eut(L)\\ \eul\end{array}\right).\]
The boundary terms can be written in the form
\[
A_{\mathcal{H}}= \mathcal{C}_x^0 \,\iup_{0}(x_{0}')+\mathcal{C}_u^0 \,\iup_{0}(u_{0}')+\mathcal{C}^1_x\, \es(\iup_{0}(x_{0}')) + \mathcal{C}_u^1\, \es(\iup_{0}(u_{0}')),
\]
where
\[
 \begin{array}{rcl}
\mathcal{C}_x^0&=& \es_{-1} \left\{ \co\, \eul(L)-\,\ell^{-1}\si \eut(L) +\ell^{-1}\si\,\es_{-1}\left(\eut(L) \right) \right\},\\[5pt] 
\mathcal{C}_u^0&=&\es_{-1} \left\{ \si \eul(L)+ \left(\es\left(\ell^{-1}\right)+\ell^{-1}\co-\,\ell^{-1}\co\,\es_{-1}\right)\eut(L) \right\},\\[5pt]
\mathcal{C}_x^1&=&\ell^{-1}\si\, \es_{-1}\left\{\eut(L) \right\},\\[5pt]
\mathcal{C}_u^1&=&-\,\ell^{-1}\co\, \es_{-1}\left\{\eut(L) \right\}.
\end{array}
\]
The infinitesimal vector fields are
$$ \mathbf{v}_a=\partial_{x_0},\qquad \mathbf{v}_b=\partial_{u_0},\qquad \mathbf{v}_{\theta}= -u_0\partial_{x_0}+x_0\partial_{u_0},$$
so that
$$\Phi(\mbu_0)=\begin{pmatrix}1& 0& u_0\\ 0 & 1&-x_0\end{pmatrix},\qquad \Phi(I)=\begin{pmatrix}1& 0& 0\\ 0 & 1&0\end{pmatrix},$$
and
$$\myAd\big(g(\theta,a,b)\big)=\begin{pmatrix}\cos\theta & -\sin\theta  & b\\ \sin\theta & \cos\theta&-a\\ 0&0&1\end{pmatrix}.
$$

Applying the replacement (\ref{InvAdrepEqn}) and simplifying and collecting terms, the conservation laws can be written in terms of the row vector of invariants as follows
\begin{equation}\label{EulerElasticaCLs}
(V_1\ V_2\ V_3)\left(
\begin{array}{cc}
R_{\theta_{0}} & J R_{\theta_0}\mbu_0 \\
0 & 1
\end{array}
\right)=(c_1\ c_2\ c_3),\qquad J=\left(\begin{array}{cc} 0&-1\\ 1&0\end{array}\right),
\end{equation}
where
\[
\begin{array}{l}
V_1=\es_{-1}(\cos\hth\eul (L)) + \left\{\es_{-1}\left(\ell^{-1}\sin\hth\right)(\es_{-2}-\es_{-1})\right\}\eut (L),\\[5pt]
V_2=\es_{-1}(\sin\hth\eul (L)) -\left\{
\es_{-1}\left(\ell^{-1}\co \right)(\es_{-2} - \es_{-1})\right\}\eut (L),\\[5pt] 
V_3=-\es_{-1}\left(\eut(L) \right).
 \end{array} 
\]

Using Maple, we solve the discrete Euler--Lagrange equations for the invariants as an initial data problem. 
Note that (\ref{EulerElasticaCLs}) implies that $$\left(V_1\right)^2 +\left(V_2\right)^2=\left(c_1\right)^2 +\left(c_2\right)^2,$$ which gives a first integral of the discrete Euler--Lagrange equations.
Further, there is a linear relation for $x$ and $u$, in terms of the invariants; using the methods of \S\ref{methodforEulerElas},
the solution in terms of the original variables can be obtained in a straightforward manner. The initial data give the values of the constants $c_1$, $c_2$.
We have used these constants and the initial values $(x_0,u_0)=(0,1)$ to obtain the initial data for the smooth solution.
The discrete equations require one more initial datum than the smooth, so that more than one discrete solution will have the same constants and starting point, and hence more than one discrete solution can approximate a given smooth one. In Figure \ref{EulerSols}, we compare two discrete solutions with differing initial step sizes, both
approximating the single smooth solution. Magnifications of these solutions, {verification that conserved quantities are indeed conserved, and relative errors of each discrete solution to the smooth solution} are given Figures \ref{FigAmps} and \ref{FigConsRelErr}.

\begin{figure}[t]
	\begin{center}
		\begin{tabular}{c}
			\includegraphics[scale=0.5]{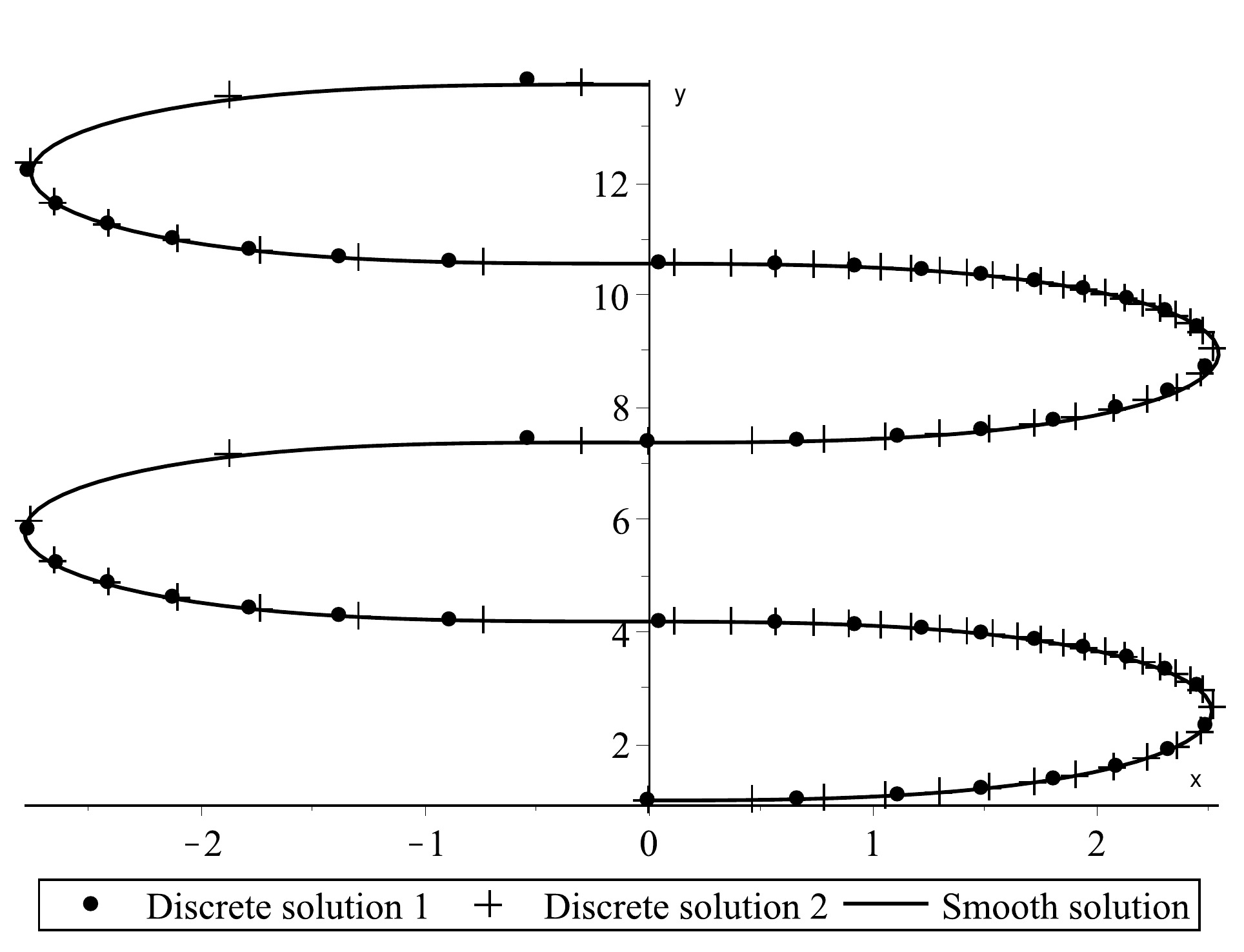}
		\end{tabular}
	\end{center}
	\caption{\label{EulerSols} A plot of  {{an extract}} of 847 {{points}} of the discrete solution for certain initial data and {{an extract}} of 507 {{points}} of the discrete solution for a variation of the previous initial data. This is compared with an accurate numerical solution of  {{the third column of \eqref{EEconlaws}, \eqref{EEfirstint} and \eqref{EExseq}}},  using a Fehlberg fourth-fifth order Runge-Kutta method with degree four interpolant, with uniform step $0.1$. {{The conservation laws are used in the solution in order to match the initial data.}}}\label{FigElastica1}
\end{figure}

\begin{figure}[!h]
	\begin{tabular}{cc}
		\hspace*{-12pt}\includegraphics[scale=0.38]{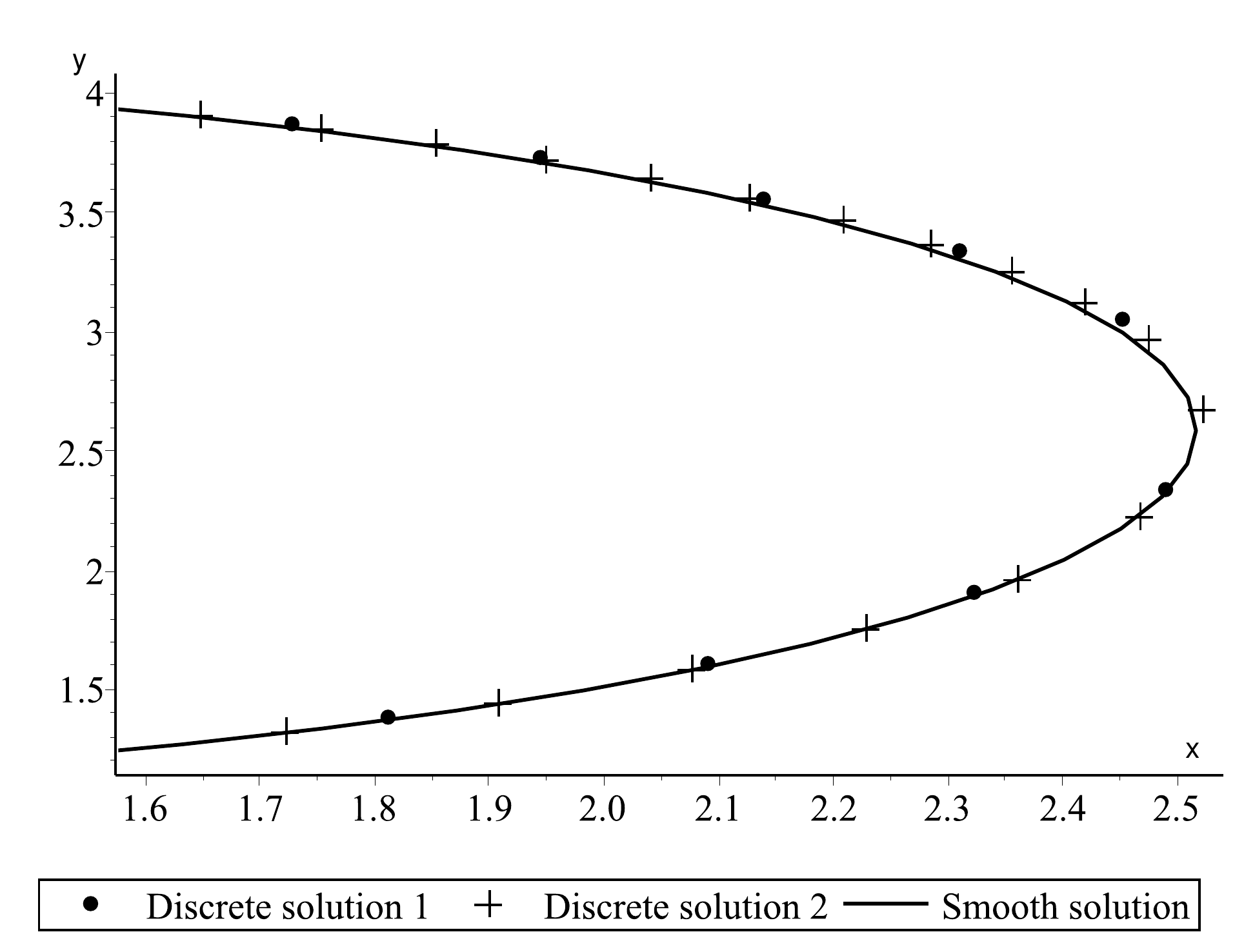} & \includegraphics[scale=0.38]{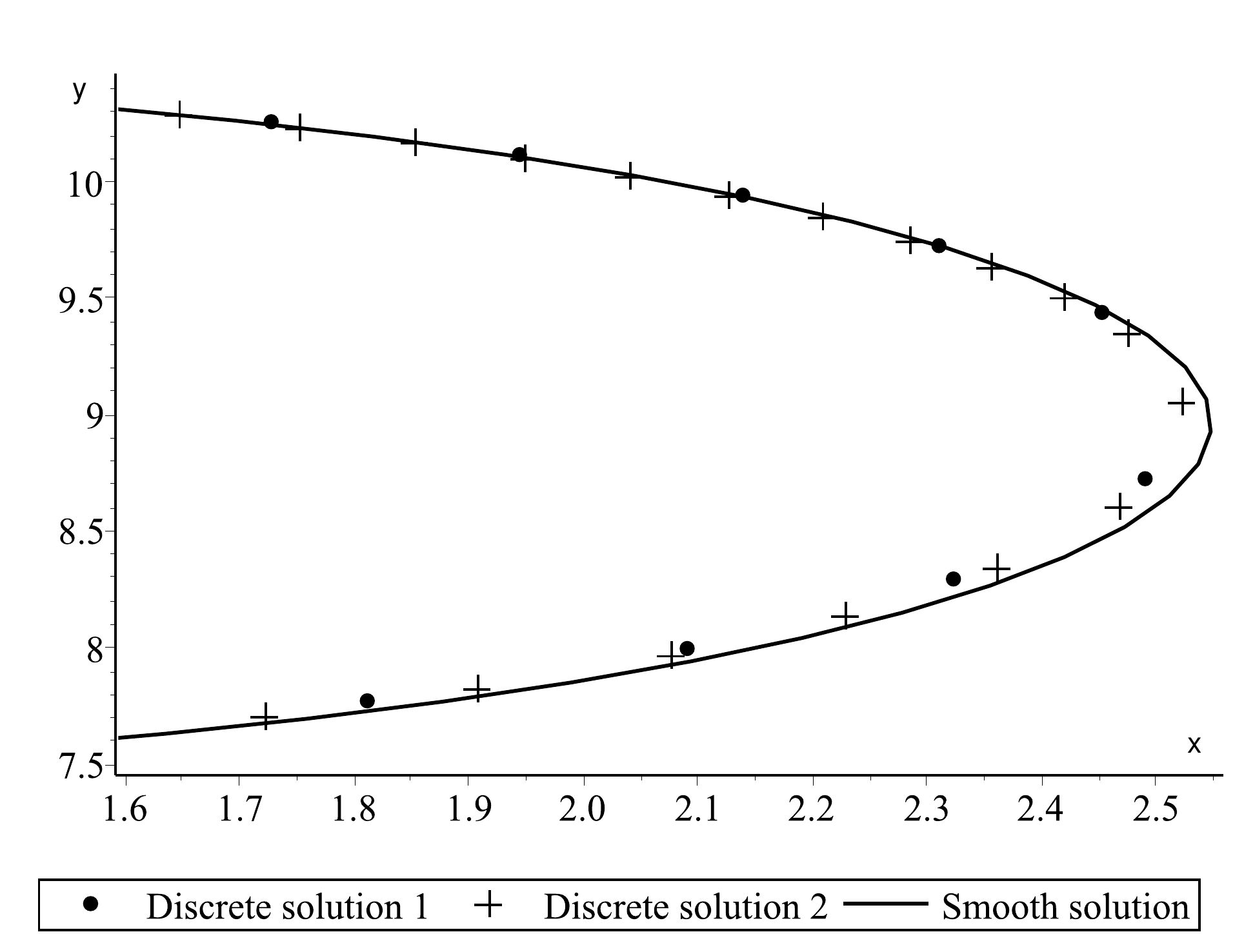}\\
		{ {\scriptsize  (a)}} & { {\scriptsize  (b)}} \\
	\end{tabular}
	\caption{\label{FigAmps}Plots (a) and (b) magnify two regions of Figure 2.}\label{FigElastica2}
\end{figure}

\begin{figure}[t]
	\begin{center}
		\begin{tabular}{cc}
			\includegraphics[scale=0.35]{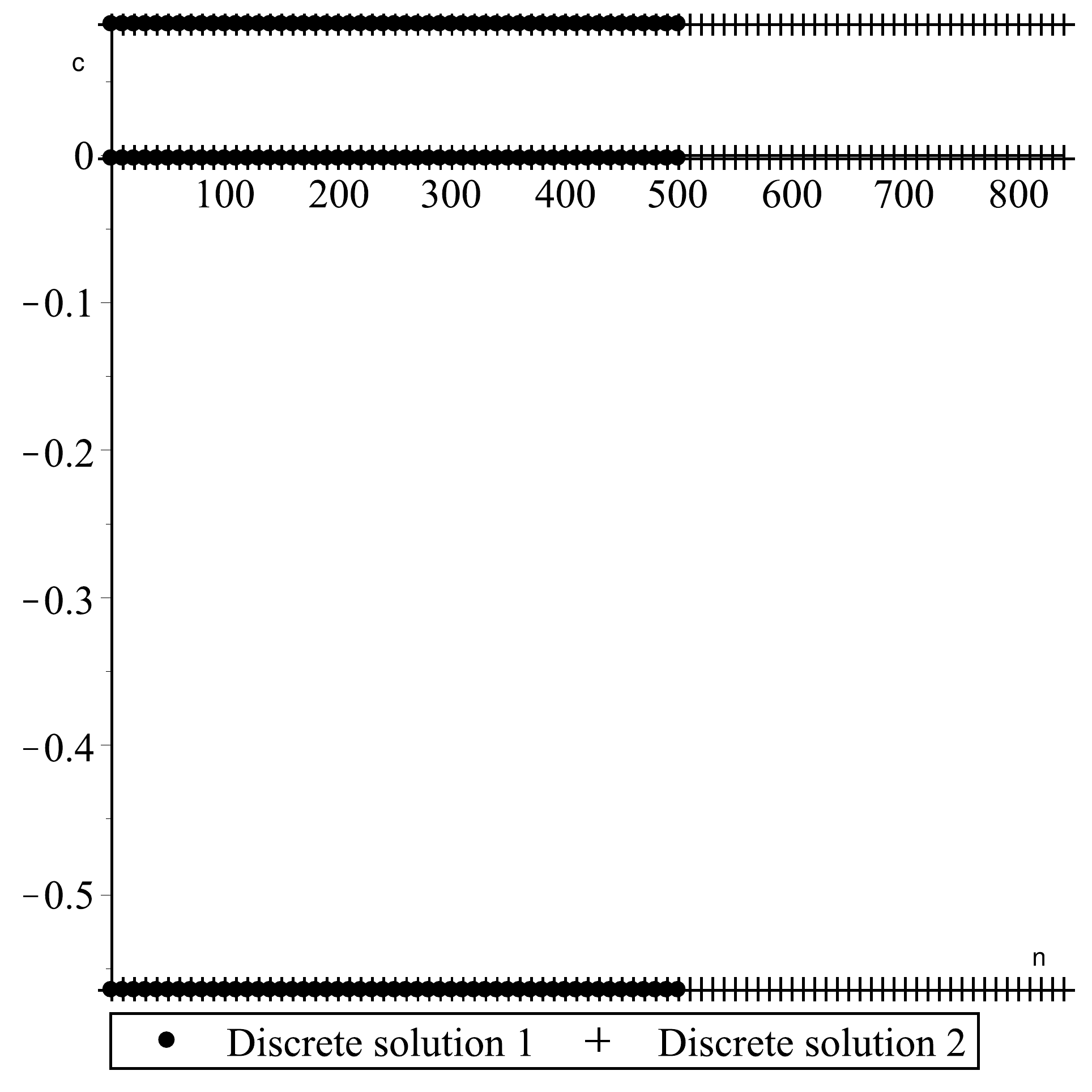} & \includegraphics[scale=0.35]{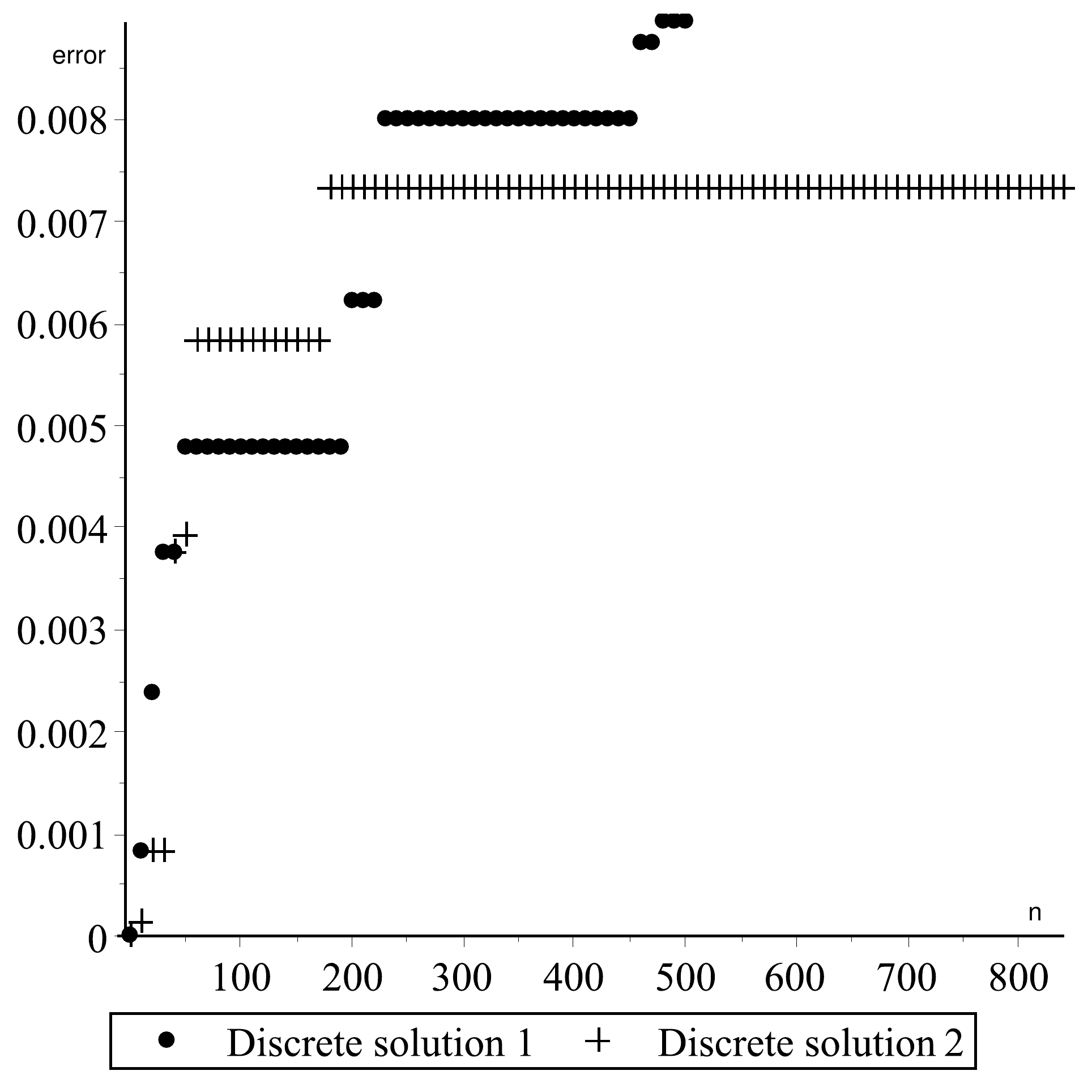}\\
			{ {\scriptsize  (a)}} & { {\scriptsize  (b)}} \\
		\end{tabular}
	\end{center}
	\caption{\label{FigConsRelErr} (a) {{Graph showing the three conserved quantities are indeed conserved (are constants). The values of the constants in the conservation laws obtained 
		in the discrete case were used to calculate the initial data in the integration of the smooth case.}} 
		(b) The norm relative error between {the  solution 
		of the smooth problem and each solution of the discrete problem}.}\label{FigElastica3}
\end{figure}

{{A routine albeit lengthy calculation shows that the Euler--Lagrange equations and the conservation laws converge to the smooth counterparts
in the continuum limit. }}
The relative success of this simple example shows that this approach to obtaining symmetry-preserving variational integrators via the difference moving frame merits further research. More sophisticated methods to derive discrete Lagrangians using interpolation are also being explored \citep{GMBMAN}.

\section{Conclusions}\label{conclusions}

In this paper, we have introduced difference moving frames and some applications. In particular, for a discrete Lagrangian with a Lie group of variational symmetries, a difference moving frame expresses the Euler--Lagrange equations in terms of the invariants and Noether's
conservation laws in terms of the frame and a vector of invariants. This makes explicit the equivariance of the conservation laws. The difference frame formulation can allow one to solve for the solutions in terms of the original variables by a divide-and-conquer approach: first the invariantized Euler--Lagrange equations are solved, then the conservation laws and frame are used to construct the complete solution. It is worth noting that one need not use the full Lie group to do this; a Lie subgroup may do the job more simply. For instance, using the translation subgroup in our running example would have been equally effective, with slightly simpler calculations. 

We have developed the Noether theory for difference frames to cover all variational symmetries, not just those that leave the given Lagrangian invariant. It also covers symmetries whose generators are $n$-dependent and Lagrangians that vary with $n$.

{{We have shown that, by matching the smooth and difference frames for the smooth and discretized problems respectively, it is possible to create symmetry-preserving variational integrators which can approximate the full set of conservations}}. Optimizing the use of the difference frame in these approximations is a subject for further research.

Part I of this paper has concentrated on relatively straightforward Lie group actions. In Part II, we turn to actions of the semisimple Lie group $\es L(2)$.

\section*{Acknowledgements}
The authors would like to thank the SMSAS at the University of Kent and the EPSRC (grant EP/M506540/1) for funding this research.

\bibliographystyle{agsm}

\clearpage
\appendix

{{
\section*{Appendix A.\  From structure constants to $\myAd(g)$}

There is an alternative way to construct the matrix $\myAd(g)$ from the basis
\[
\mbv_r=\xi^\alpha_r(\mbu)\,\partial_{u^\alpha},\qquad r=1,\dots,R,
\]
for the Lie algebra of infinitesimal generators. The Lie algebra is closed, giving rise to the commutator relations
\[
[\mbv_i,\mbv_j]\equiv \mbv_i\mbv_j-\mbv_j\mbv_i=c_{ij}^k\mbv_k,\qquad i,j=1,\dots,R,
\]
where $c_{ij}^k$ are the structure constants. Define the $R\times R$ matrices $C_j$ by
\[
(C_j)_i^k=c_{ij}^k.
\]

Every infinitesimal generator $\mbv=a^r\mbv_r$ can be exponentiated to produce a one-parameter local Lie group of transformations $\exp(\varepsilon\mbv)$. These act on an arbitrary smooth function $F(\mbu)$ as follows:
\[
\exp(\varepsilon\mbv)F(\mbu)=F(\widehat\mbu(\varepsilon)),\qquad\text{where}\quad \widehat\mbu(\varepsilon)=\exp(\varepsilon\mbv)\mbu\equiv\sum_{m=0}^\infty\frac{\varepsilon^m}{m!}\,\mbv^m(\mbu).
\]
One can obtain $\widehat{\mbu}(\varepsilon)$ by solving the initial value problem
\[
\frac{\upd \widehat{u}^{\,\alpha}(\varepsilon)}{\upd \varepsilon}=a^r\xi_r^\alpha(\widehat{\mbu}(\varepsilon)),\quad\alpha=1,\dots,q;\qquad \widehat{\mbu}\big|_{\varepsilon=0}=\mbu.
\]

\begin{theorem}\label{adthm}
	Suppose that $g\cdot \mbu=\exp(a^r\mbv_r)\mbu$. Then $\myAd(g)=\exp(a^rC_r)$.
\end{theorem}

\begin{proof}
Let $\mbv=a^r\mbv_r$ and consider the action of the one parameter local Lie group $\exp(\varepsilon\mbv)$. Define
\[
g(\varepsilon)\cdot\mbu=\widehat\mbu(\varepsilon)=\exp(\varepsilon\mbv)\mbu
\]
and denote the components of $\myAd(g(\varepsilon))$ by $A_i^l(\varepsilon)$.
The identity \eqref{addef} yields $$\mbv_i=A_i^l(\varepsilon)\widehat{\mbv}_l(\varepsilon),$$ where $\widehat{\mbv}_l(\varepsilon)$ is $\mbv_l$ with $\mbu$ replaced by $\widehat{\mbu}(\varepsilon)$. If $F(\mbu)$ an arbitrary smooth function, 
\[
\exp(-\varepsilon\mbv)\,\mbv_i\exp(\varepsilon\mbv)F(\mbu)=\exp(-\varepsilon\mbv)\left\{A_i^l(\varepsilon)\widehat{\mbv}_l(\varepsilon)F(\widehat\mbu(\varepsilon))\right\}=A_i^l(\varepsilon)\mbv_lF(\mbu).
\]
Consequently,
\[
\frac{\upd}{\upd\varepsilon}\left\{A_i^l(\varepsilon)\right\}\mbv_lF(\mbu)
=\exp(-\varepsilon\mbv)\,[\mbv_i,\mbv]\exp(\varepsilon\mbv)F(\mbu)
=a^jc_{ij}^kA_k^l(\varepsilon)\mbv_lF(\mbu),
\]
and so $\myAd(g(\varepsilon))$ satisfies the initial value problem
\[
\frac{\upd}{\upd\varepsilon}\,\myAd(g(\varepsilon))=a^j\,\myAd(g(\varepsilon))C_j,\qquad \myAd(g(0))=\text{Id}.
\]
The $R\times R$ matrix $\exp(\varepsilon a^jC_j)$ satisfies the same initial value problem, so the standard uniqueness theorem for ODEs proves that
\[
\myAd(g(\varepsilon))=\exp(\varepsilon a^jC_j).
\]
Set $\varepsilon=1$ to conclude the result.
\end{proof}

\begin{corollary}\label{adcor}
	Suppose that $$g\cdot \mbu=\exp(a_1^{r_1}\mbv_{r_1})\exp(a_2^{r_2}\mbv_{r_2})\cdots\exp(a_M^{r_M}\mbv_{r_M})\mbu.$$ Then $\myAd(g)=\exp(a_1^{r_1}C_{r_1})\exp(a_2^{r_2}C_{r_2})\cdots\exp(a_M^{r_M}C_{r_M})$.
\end{corollary}
\begin{proof}
	The adjoint representation preserves left multiplication.
\end{proof}

For every $R$-dimensional local Lie group of transformations, the action of every $g$ in some neighbourhood of the identity can be written as $\exp(a^r\mbv_r)$ for some $a^r$, which are called canonical coordinates of the first kind. So if $\widetilde{\mbu}=g\cdot\mbu$ is written in this form, Theorem \ref{adthm} enables us to write down $\myAd(g)=\exp(a^rC_r)$ immediately.

Commonly, $\widetilde{\mbu}=g\cdot\mbu$ is written as a product of exponentiated generators, in which case Corollary \ref{adcor} is applicable. In particular, the action of every $g$ in some neighbourhood of the identity can be written as $\exp(a^R\mbv_R)\cdots\exp(a^1\mbv_1)$; here $a^r$ are canonical coordinates of the second kind.

\begin{example}(Example \ref{simpleExOne} cont.)\  
	The group action $g\cdot(x,u)=(\tilde{x},\tilde{u})$ in the running example can be written in terms of canonical coordinates of the second kind as
	\[
	(\tilde{x},\tilde{u})=\left(e^{3a^1}\!x+a^2,\,e^{a^1}\!u+a^3\right)=\exp(a^3\mbv_3)\exp(a^2\mbv_2)\exp(a^1\mbv_1)(x,u),
	\]
	where $(a^1,a^2,a^3)=(\ln(\lambda),a,b)$ and
	\[
	\mbv_1=3x\partial_{x}+u\partial_{u},\qquad\mbv_2=\partial_x,\qquad\mbv_3=\partial_u.
	\]
	The only nonzero structure constants are $c_{12}^2=-c_{21}^2=-3$ and $c_{13}^3=-c_{31}^3=-1$, so Corollary \ref{adcor} yields
	\[
	\myAd(g)=\exp\!\begin{pmatrix}0&\!0&\!0\\0&\!0&\!0\\-a^3&\!0&\!0\end{pmatrix}\!\exp\!\begin{pmatrix}\!0&\!0&\!0\\\!-3a^2&\!0&\!0\\\!0&\!0&\!0\end{pmatrix}\!\exp\!\begin{pmatrix}0&\!0&\!0\\0&\!3a^1&\!0\\0&\!0&\!a^1\end{pmatrix}\!=\!\begin{pmatrix}\!1&\!0&\!0\\\!-3a^2&\!e^{3a^1}&\!0\\\!-a^3&\!0&\!e^{a^1}\end{pmatrix}\!,
	\]
	in agreement with the more straightforward approach of Section \ref{moregrpsec}.
\end{example}
}}

\end{document}